\documentclass[12pt,reqno]{amsart}
\usepackage{amsmath,amsthm,amssymb}
\usepackage{url}
\usepackage{xparse}
\usepackage{mathtools}
\usepackage{hyperref}
\usepackage[noabbrev,capitalize]{cleveref}
\usepackage{xcolor}
\definecolor{darkcandyapplered}{rgb}{0.64, 0.0, 0.0}
\definecolor{midnightblue}{rgb}{0.1, 0.1, 0.44}
\hypersetup{
    pdftoolbar=true,        
    pdfmenubar=true,        
    pdffitwindow=false,     
    pdfstartview={FitH},    
    pdftitle={},
    pdfauthor={},
    pdfsubject={},
    pdfkeywords={},
    pdfnewwindow=true,      
    colorlinks=true,       
    linkcolor=darkcandyapplered,          
    citecolor=midnightblue,        
    urlcolor=cyan,           
    linktocpage=true   
    }

\usepackage{color}
\theoremstyle{plain}
   \newtheorem{theorem}{Theorem}[section]
   
   \newtheorem{lemma}[theorem]{Lemma}
   \newtheorem{corollary}[theorem]{Corollary}

\theoremstyle{definition}

\theoremstyle{remark}

\numberwithin{equation}{section}
\newcommand{\NN}{\mathbb{N}}
\newcommand{\ZZ}{\mathbb{Z}}

\newcommand{\CC}{\mathbb{C}}
\newcommand{\fS}{\mathfrak{S}}
\newcommand{\fB}{\mathfrak{B}}
\newcommand{\aA}{\mathcal{A}}
\newcommand{\bB}{\mathcal{B}}
\newcommand{\cC}{\mathcal{C}}
\newcommand{\dD}{\mathcal{D}}
\newcommand{\iI}{\mathcal{I}}
\newcommand{\bx}{{\mathbf x}}
\newcommand{\bX}{{\mathbf X}}
\newcommand{\by}{{\mathbf y}}

\newcommand{\bl}{{\pmb\lambda}}
\newcommand{\bm}{{\pmb\mu}}
\newcommand{\bQ}{{\pmb Q}}
\newcommand{\bP}{{\pmb P}}
\renewcommand{\to}{\rightarrow}
\newcommand{\toto}{\Rightarrow}

\newcommand{\ol}{\overline}
\newcommand{\Om}{\Omega}

\renewcommand{\emptyset}{\varnothing}

\newcommand{\ssum}{{\rm sum}}

\def\newop#1{\expandafter\def\csname #1\endcsname{\mathop{\rm
#1}\nolimits}}

\newop{Des}
\newop{iDes}
\newop{lDes}
\newop{sDes}
\newop{nDes}
\newop{NDes}
\newop{des}
\newop{ides}
\newop{ifdes}
\newop{ldes}
\newop{ndes}
\newop{fdes}
\newop{Asc}
\newop{asc}
\newop{Exc}
\newop{exc}
\newop{wexc}
\newop{fexc}
\newop{inv}
\newop{finv}
\newop{maj}
\newop{imaj}
\newop{ifmaj}
\newop{fmaj}
\newop{nmaj}
\newop{rmaj}
\newop{lmaj}
\newop{Neg}
\newop{neg}
\newop{nsum}
\newop{csum}
\newop{col}
\newop{fix}
\newop{Den}
\newop{den}
\newop{fden}
\newop{abs}
\newop{flag}
\newop{stat}
\newop{istat}

\newop{eul}
\newop{mah}

\newop{SYT}
\newop{SYB}
\newop{SSYT}
\newop{Stab}
\newop{ch}
\newop{Sym}
\newop{QSym}
\newop{St}
\newop{ct}
\newop{ps}
\newop{c}
\newop{h}
\newop{b}

\ExplSyntaxOn
\NewDocumentCommand{\cycle}{ O{\;} m }
 {
  (
  \alec_cycle:nn { #1 } { #2 }
  )
 }

\seq_new:N \l_alec_cycle_seq
\cs_new_protected:Npn \alec_cycle:nn #1 #2
 {
  \seq_set_split:Nnn \l_alec_cycle_seq { , } { #2 }
  \seq_use:Nn \l_alec_cycle_seq { #1 }
 }
\ExplSyntaxOff

\makeatletter
\renewcommand{\tocsection}[3]{%
  \indentlabel{\@ifnotempty{#2}{\bfseries\ignorespaces#1 #2\quad}}\bfseries#3}
\renewcommand{\tocsubsection}[3]{%
  \indentlabel{\@ifnotempty{#2}{\ignorespaces#1 #2\quad}}#3}

\newcommand\@dotsep{4.5}
\def\@tocline#1#2#3#4#5#6#7{\relax
  \ifnum #1>\c@tocdepth 
  \else
    \par \addpenalty\@secpenalty\addvspace{#2}%
    \begingroup \hyphenpenalty\@M
    \@ifempty{#4}{%
      \@tempdima\csname r@tocindent\number#1\endcsname\relax
    }{%
      \@tempdima#4\relax
    }%
    \parindent\z@ \leftskip#3\relax \advance\leftskip\@tempdima\relax
    \rightskip\@pnumwidth plus1em \parfillskip-\@pnumwidth
    #5\leavevmode\hskip-\@tempdima{#6}\nobreak
    \leaders\hbox{$\m@th\mkern \@dotsep mu\hbox{.}\mkern \@dotsep mu$}\hfill
    \nobreak
    \hbox to\@pnumwidth{\@tocpagenum{\ifnum#1=1\bfseries\fi#7}}\par
    \nobreak
    \endgroup
  \fi}
\AtBeginDocument{%
\expandafter\renewcommand\csname r@tocindent0\endcsname{0pt}
}
\def\l@subsection{\@tocline{2}{0pt}{2.5pc}{5pc}{}}
\makeatother

\begin{document}
\title[Euler--Mahonian statistics via quasisymmetric functions]{Specializations of colored quasisymmetric functions and Euler--Mahonian identities}
\author{Vassilis-Dionyssis~Moustakas}
\address{Department of Mathematics\\
National and Kapodistrian University of Athens\\
Panepistimioupolis\\
15784 Athens, Greece}
\email{vasmous@math.uoa.gr}
\date{\today}
\begin{abstract}
We propose a unified approach to prove general formulas for the joint distribution of an Eulerian and a Mahonian statistic over a set of colored permutations by specializing Poirier's colored quasisymmetric functions. We apply this method to derive formulas for Euler--Mahonian distributions on colored permutations, derangements and involutions. A number of known formulas are recovered as special cases of our results, including formulas of Biagioli--Zeng, Assaf, Haglund--Loehr--Remmel, Chow--Mansour, Biagioli--Caselli, Bagno--Biagioli, Faliharimalala--Zeng. Several new results are also obtained. For instance, a two-parameter flag major index on signed permutations is introduced and formulas for its distribution and its joint distribution with some Eulerian partners are proven.
\end{abstract}
\maketitle
%
%
\tableofcontents

%
%
\section{Introduction}
\label{sec: intro}

	For a positive integer $n$, we denote by $\fS_n$ the set of permutations of $[n] := \{1, 2, \dots, n\}$. For $w \in \fS_n$, an index $i \in [n-1]$ is called a \emph{descent} of $w$, if $w(i) > w(i+1)$. The set of all descents of $w$, written $\Des(w)$, is called the \emph{descent~set} of $w$. The cardinality and the sum of all elements of $\Des(w)$ are written as $\des(w)$ and $\maj(w)$, respectively, and called the \emph{descent~number} and \emph{major~index} of $w$. A statistic on $\fS_n$ which is equidistributed with $\des$ (resp. $\maj$) is called \emph{Eulerian} (resp. \emph{Mahonian}). Let
\[
A_n (x, q) \ := \ \sum_{w \in \fS_n} \, x^{\des(w)} q^{\maj(w)}
\]
be the generating polynomial for the joint distribution $(\des, \maj)$ on $\fS_n$, sometimes called the $n$-th \emph{~$q$-Eulerian~polynomial}. The polynomial $A_n (x) := A_n (x, 1)$ is called  the $n$-th \emph{Eulerian~polynomial} and constitutes one of the most important polynomials in combinatorics. The interested reader is referred to the wonderful book exposition of Petersen \cite{PetEN}, and references therein, for many properties of Eulerian distributions and connections with algebra and geometry.

	The distribution of a pair $(\eul, \mah)$ of permutation statistics on $\fS_n$, which satisfies 
\[
\sum_{w \in \fS_n} \, x^{\eul(w)}q^{\mah(w)} \ = \ A_n(x, q)
\]
is called \emph{Euler--Mahonian}. MacMahon \cite[Vol.2,~Section~IX]{MacCA} proved a formula which specializes to
\begin{equation}
\label{eq: Carlitz}
\sum_{m \ge 0} \, [m + 1]_q^n \, x^m \ = \ \frac{A_n (x, q)}{(1-x)(1-xq)\cdots (1-xq^n)},
\end{equation}
where $A_0 (x, q) := 1$ and $[n]_{q} := 1 + q + \cdots + q^{n-1}$ is the $q$-analogue of $n$. This formula, usually attributed to Carlitz \cite{Car75} (and hence called the \emph{Carlitz~identity}), serves as the basis for many generalizations on Coxeter groups and $r$-colored permutation groups, that is wreath products $\ZZ_r \wr \fS_n$. Identities such as \eqref{eq: Carlitz}, involving Euler--Mahonian distributions will be called \emph{Euler--Mahonian~identities}. \Cref{eq: Carlitz}, for $q=1$, reduces to the following identity \cite[Proposition~1.4.4]{StaEC1} (sometimes used as the definition of Eulerian polynomials) 
\begin{equation}
\label{eq: Euler}
\sum_{m \geq 0} \, (m+1)^n x^m  \ = \ \frac{A_n(x)}{(1 - x)^{n+1}}.
\end{equation}

	As mentioned in \cite[Section~1]{BB13}, a general approach to prove Euler--Mahonian identities, among others, is via the theory of symmetric/quasisymmetric functions. Quasisymmetric functions are certain power series in infinitely many variables that generalize the notion of symmetric functions. They first appeared, not with this name yet, in Stanley's thesis \cite{Sta72} (for a detailed description of Stanley's contribution to symmetric/quasisymmetric functions see \cite{BM16}) and were later defined and studied systematically by Gessel \cite{Ges84} (see also \cite[Section~7.19]{StaEC2} and \cite[Section~8.5]{Ges16}). In this paper, we aim to provide a unified approach to prove Euler--Mahonian identities on $r$-colored permutation groups by specializing Poirier's colored analogue of quasisymmetric functions \cite{Poi98}.

	Specializations of symmetric functions date back to Stanley's work on the enumeration of plane partitions \cite{Sta71}. Gessel and Reutenauer, in their seminal paper \cite{GR93}, used the stable principal specialization and the principal specialization of order $m$ of fundamental quasisymmetric functions, together with the fact that the quasisymmetric generating function of the set of permutations of fixed cycle type is symmetric, to derive formulas for the joint distribution of the descent statistic and major index on cycles, involutions and derangements. Let us now illustrate how one can specialize fundamental quasisymmetric functions in order to prove \cref{eq: Carlitz}. This proof will serve as a prototype for all proofs of our applications in \cref{sec: applications}. For a similar approach, see the recent paper of Gessel and Zhuang \cite{GZ20}.

	For a sequence $\bx = (x_1, x_2, \dots)$ of commuting indeterminates, the  \emph{fundamental~quasisymmetric~function} associated to $S \subseteq [n-1]$ is defined by  
\begin{equation} \label{eq:defFS(x)}
  F_{n, S} (\bx) \ :=
  \sum_{\substack{ i_1 \ge i_2 \ge \cdots \ge i_n \ge 1  \\ j \in S \,\Rightarrow\, 
  i_j < i_{j+1}}}
  x_{i_1} x_{i_2} \cdots x_{i_n},
\end{equation}
and $F_{0, \emptyset} (\bx) := 1$. The original definition \cite[Equation~(7.89)]{StaEC2} actually defines $F_{n, S}$ as in \eqref{eq:defFS(x)} with $1 \le i_1 \le i_2 \le \cdots \le i_n$ instead. Our choice of definition will become clear later in the paper (see also \cite[Section~4.3]{Ath18}). The standard way to connect quasisymmetric functions with permutation statistics is by associating the fundamental quasisymmetric function $F_{n, \Des(w)}$ with $w \in \fS_n$ as done in \cite[Section~5]{GR93}. Let $\CC[q]$ (resp. $\CC[[q]]$) denote the ring of polynomials (resp. formal power series) in variable $q$ with complex coefficients. The \emph{stable~principal~specialization} (resp. \emph{principal~specialization} of order $m$) is a ring homomorphism $\ps : \QSym (\bx) \to \CC[[q]]$ (resp. $\ps_m :\QSym (\bx) \to \CC[q]$) defined by the substitutions
\[
x_1 = 1, \, x_2 = q, \, x_3 = q^2, \, \dots
\]
and
\[
x_1 = 1, \, x_2 = q, \, \dots, x_m = q^{m-1}, \, x_{m+1} = x_{m+2} = \cdots = 0,
\]
respectively, for a positive integer $m$, where $\QSym (\bx)$ stands for the $\CC$-algebra of quasisymmetric functions of bounded degree with complex coefficients in $\bx$.

	The principal specialization of order $m$ and stable principal specialization of $F_{n, S}$ are given by the following formulas \cite[Lemma~5.2]{GR93} (see also the first half of \cite[Section~4]{Ges84})
\begin{align}
\sum_{m \ge 1} \, \ps_m (F_{n, S}) \, x^{m-1} \ &= \
\frac{x^{|S|} \, q^{\ssum(S)}}{(1-x)(1-xq)\cdots (1-xq^n)} \label{eq: psmFS(x)} \\
\ps (F_{n, S}) \ &= \ \frac{q^{\ssum(S)}}{(1-q)(1-q^2) \cdots (1 - q^n)}. \label{eq: psFS(x)}
\end{align}
where $\ssum(S)$ stands for the sum of all elements of $S$. These formulas allow us to study the Euler--Mahonian distribution on $\fS_n$ by specializing the quasisymmetric generating function
\[
F(\aA; \bx) \ := \ \sum_{w \in \aA} \, F_{n, \Des(w)}(\bx)
\]
associated to a subset $\aA \subseteq \fS_n$.  In particular, one has \cite[Theorem~5.3]{GR93}
\begin{align}
\sum_{m \ge 1} \, \ps_m ( F (\aA; \bx) ) \, x^{m-1} \ &= \
\frac{\sum_{w \in \aA} \, x^{\des(w)} \, q^{\maj(w)}}{(1-x)(1-xq)\cdots (1-xq^n)} \label{eq: psmF(A;x)} \\
\ps ( F (\aA; \bx) ) \ &= \ \frac{\sum_{w \in \aA} \, q^{\maj(w)}}{(1-q)(1-q^2) \cdots (1 - q^n)}, \label{eq: psF(A;x)}
\end{align}

	A celebrated result, due to MacMahon (see, for example, \cite[Chapter~1,~Notes]{StaEC1}), is the formula
\begin{equation}
\label{eq: MacMahon}
A_n (1, q) \ = \ [1]_q[2]_q\cdots [n]_q
\end{equation}
for the distribution of the major index on $\fS_n$. Although the relation between Equations \eqref{eq: Carlitz} and \eqref{eq: MacMahon} is not obvious, the above mentioned machinery allows us to easily prove both these equations in a uniform way. The quasisymmetric generating function $F(\fS_n; \bx)$ is known to have \cite[Corollary~7.12.5]{StaEC2} the following nice form
\begin{equation}
\label{eq: F(Sn,x)}
F(\fS_n; \bx) = (x_1 + x_2 + \cdots)^n.
\end{equation}
Taking the principal specialization of order $m$  and the stable principal specialization of Equation \eqref{eq: F(Sn,x)} yields 
\begin{align*}
\ps_m ( F (\fS_n; \bx) ) \ &= \ [m]_q^n \\
\ps (F(\aA; \bx) ) \ &= \ \frac{1}{(1 - q)^n},
\end{align*}
respectively. Then, Equations \eqref{eq: Carlitz} and \eqref{eq: MacMahon} follow by substituting these computations in Equations \eqref{eq: psmF(A;x)} and \eqref{eq: psF(A;x)}, for $\aA = \fS_n$, respectively.

	This proof suggests that whenever $F(\aA; \bx)$ has a nice form, then we can use Formulas \eqref{eq: psmF(A;x)} and \eqref{eq: psF(A;x)} to compute Euler--Mahonian identities on $\aA$. In particular, for $\aA = \dD_n$, the set of derangements of $\fS_n$, Gessel and Reutenauer \cite[Theorem~8.1]{GR93} computed $F(\dD_n; \bx)$ in terms of elementary and complete homogeneous symmetric functions. In Section \ref{subsec: derangem} we specialize their formula to prove an Euler--Mahonian identity on $\dD_n$, which refines Wachs' celebrated formula \cite[Theorem~4]{Wa89}
\begin{equation}
\label{eq: Wachs}
\sum_{w \in \dD_n} \, q^{\maj(w)} \ = \ [n]_q! \sum_{k = 0}^n \, (-1)^k \, \frac{q^{\binom{k}{2}}}{[k]_q!}.
\end{equation}

	The goal of this paper is to extend this symmetric/quasisymmetric function approach by specializing colored quasisymmetric functions, introduced by Poirier \cite{Poi98}, in order to prove Euler--Mahonian identities for the colored permutation groups. In other words, we provide a colored analogue of the method demonstrated above (pioneered by Gessel and Reutenauer in \cite{GR93}), by replacing Gessel's fundamental quasisymmetric functions with Poirier's colored quasisymmetric functions and using variations of the stable principal specialization and principal specialization of order $m$. A first instance of this technique appears in the work of Athanasiadis \cite[Equation~(45)]{Ath18}, where he proves \cite[Proposition~2.22]{Ath18} an expansion of the generating polynomial of the Eulerian distribution on signed involutions in terms of Eulerian polynomials of type $B$. We provide a colored generalization of this formula (see \cref{cor: colored athanasiadis}) for absolute involutions, a class of $r$-colored permutations which coincides with involutions and signed involutions for $r=1$ and $r=2$, respectively. A second instance appears in \cite[Lemma~3.1]{Mou19}. 

	Different type $B$ analogues of quasisymmetric functions have been suggested (for a comparison between Chow's type $B$ quasisymmetric functions and Poirier's, see \cite{Pet05}). Our choice of colored quasisymmetric functions, and much of the motivation behind this paper comes from the fact that Poirier's signed analogue of the fundamental quasisymmetric functions \cite{Poi98} were recently employed by Adin et al. \cite{AAER17} in order to define and study a signed analogue of the concept of fine sets and fine characters of Adin and Roichman \cite{AR15}. All classes of permutations considered in the applications are fine sets in the sense of \cite{AAER17}. 

	The paper is organized as follows. \Cref{sec: prelim} surveys colored permuation statistics and Euler--Mahonian identities known in the literature and discusses colored quasisymmetric functions.  \Cref{sec: special} presents the main results of this paper on specializations of colored quasisymmetric functions (see Theorems \ref{thm: Icolor}, \ref{thm: IIcolor} and \ref{thm: IIIcolor}). In this section we also introduce a two-parameter flag major index on signed permutations and provide formulas for its joint distribution with some Eulerian partner (\cref{thm: IVklvariation}). \Cref{sec: applications} presents applications of our main results. In particular, \cref{subsec: Euler--Mahonian identities} concerns Euler--Mahonian identities on permutations, Sections \ref{subsec: derangem} and \ref{subsec: involution} concern Euler--Mahonian identities on derangements and involutions, respectively (in the case of involutions, we also discuss some enumerative aspects). \Cref{subsec: bimahonian} concerns bimahonian and multivariate distributions involving Eulerian and Mahonian statistics on colored permutations.  We recover, as special cases of our results, various formulas of Biagioli--Zeng (\cref{eq: applicationIcolorChowGesselmaj}), Assaf (Equations \eqref{eq: applicationIcolormajtypeA} and \eqref{eq: assaf}), Haglund--Loehr--Remmel (\cref{eq: applicationIIcolorfmaj}), Biagioli--Caselli, Chow--Mansour and Biagioli--Zeng (\cref{eq: applicationIIcolordesfmaj}), Bagno--Biagioli (\cref{eq: applicationIIIcolorfdesfmaj}) and Faliharimalala--Zeng (\cref{eq: FaliharimalalaZengfmaj}). We refine and generalize formulas of  Chow--Gessel (\cref{eq: applicationIVdesfmaj}), Adin--Roichman (\cref{eq: applicationIVfmaj}), Wachs (Theorems \ref{thm: Euler--Mahonianderangements}, \ref{thm: Euler--Mahoniancoloredderangements} and \cref{eq: desklmajoronsignedderangements}), Chow (Equations  \eqref{eq: desklmajoronsignedderangements} and \eqref{eq: klfmajonsignedderangements}), Faliharimalala--Zeng (\cref{thm: Euler--Mahoniancoloredderangements}), D\'esarm\'enien--Foata and Gessel--Reutenauer (\cref{cor: eulermahonian for colored involutions}), Athanasiadis (\cref{cor: colored athanasiadis}), Gessel, Gordon and Roselle (Equations \eqref{eq: maj imaj}, \eqref{eq: fmaj ifmaj} and \eqref{eq: fmajkl ifmajk'l'}), Garsia--Gessel (Equations \eqref{eq: des ides maj imaj}, \eqref{eq: des ides fmaj ifmaj}, \eqref{eq: fdes ifdes fmaj ifmaj} and \eqref{eq: des ides fmajkl ifmajk'l'}) and Foata--Han (\cref{cor: IIImultivariate}).

%
%
\section{Preliminaries}
\label{sec: prelim}

This section provides key definitions, discusses colored permutation statistics and surveys several known Euler--Mahonian identities for the colored permutation groups. It also recalls the definition and basic facts about Poirier's  quasisymmetric functions and reviews some tools from the representation theory of colored permutation groups, which will be used in the sequel.  

Throughout the paper we assume familiarity with basic concepts in the theory of symmetric functions as presented in \cite[Section~7]{StaEC2}. We denote by $|S|$ the cardinality of a finite set $S$, by $\ZZ$ (resp. $\NN$) the set of integers (resp. nonnegative integers). For integers $a \leq b$ we set $[a, b] := \{a, a+1, \dots, b\}$. In particular, we set $[n] :=[1, n]$. For a nonnegative integer $n$, define 
\[
(x; q)_n := \begin{cases} 1, & \text{if $n = 0$} \\
(1-x)(1-xq)\cdots(1-xq^{n-1}), & \text{if $n \ge 1$} 
\end{cases}
\]
and set $(q)_n := (q, q)_n$. Also, for a statement $P$, let $\chi(P) = 1$, if $P$ is true and $\chi(P) = 0$, otherwise. We shall use boldface to denote vectors and $r$-partite concepts.

%
%
\subsection{Colored permutation statistics}
\label{subsec: colored}

Fix a positive integer $r$ and let $\ZZ_r$ be the cyclic group of order $r$. The elements of  $\ZZ_r$, will be represented by those of $[0, r-1]$ and will be thought of as colors. Let 
\[
\Om_{n, r} := \{1^0, 2^0, \dots, n^0, 1^1, 2^1, \dots, n^1, \dots, 1^{r-1}, 2^{r-1}, \dots, n^{r-1}\}
\]
be the set of $r$-colored integers. We will often identify colored integers $i^0$ with $i$.

The \emph{$r$-colored~permutation~group}, denoted by $\fS_{n, r}$, consists of all permutations $w$ of $\Om_{n, r}$ such that $w (a^0)=b^j \toto w (a^i)=b^{i +j}$, where $i + j$ is computed modulo $r$ and the product of $\fS_{n, r}$ is composition of permutations.  The $r$-colored permutation group can be realized as the wreath product group $\ZZ_r \wr \fS_n$ (see, for example, \cite[Section~2]{Stei94}). The elements of $\fS_{n, r}$ are represented  in window notation as $w = w(1)^{c_1}w(2)^{c_2}\cdots w(n)^{c_n}$, where $w(1)w(2)\cdots w(n) \in \fS_n$ is the \emph{underlying~permutation} and $(c_1, c_2, \dots, c_n)$ is the \emph{color~vector} of $w$. We will represent both the colored permutation and the underlying permutation by the same letter. 

	The case $r=2$ is of particular interest. $\fS_{n, 2}$ is the hyperoctahedral group, written $\fB_n$, the group of signed permutations of length $n$. Signed permutations, or $2$-colored permutations, are bijective maps $w : \Om_{n, 2} \to \Om_{n, 2}$ such that $w(\ol{i}) = \ol{w(i)}$, for every $i \in \Om_{n, 2}$, where in this case the set of $2$-colored integers $\Om_{n, 2}$ is identified as the set $\{1, 2, \dots, n\} \cup \{\ol{1}, \ol{2}, \dots, \ol{n}\}$ (see, for example, \cite{AAER17}). The hyperoctahedral group is a Coxeter group of type $B_n$ (see, for example, \cite[Part~III]{PetEN}). It is often required to treat this case separately.


	Steingr\'{i}msson \cite{Stei94} studied combinatorial aspects of $r$-colored permutations by introducing a notion of descent and excedance for $r$-colored permutations and Eulerian polynomials for $\fS_{n, r}$. Let $<_{\St}$ be the  following total order on $\Om_{n, r}$
\[
1^0 <_{\St} \cdots <_{\St} n^0 <_{\St} 1^1 <_{\St} \cdots <_{\St} n^1 <_{\St} \cdots <_{\St} 1^{r-1} <_{\St} \cdots <_{\St} n^{r-1}.
\]
For $w = w(1)^{c_1}w(2)^{c_2}\cdots w(n)^{c_n} \in \fS_{n, r}$, an index $i \in [n]$ is called a \emph{descent} of $w$, if $1 \le i \le n-1$ and $w(i) >_{\St} w(i+1)$ or if $i = n$ and $c_n \neq 0$. Let $\Des_{<_{\St}}(w)$ be the set of all descents of $w$ and $\des_{<_{\St}}(w)$ its cardinality. Steingr\'{i}msson proved \cite[Theorem~17]{Stei94} the following formula for the generating polynomial of the $\des_{<_{\St}}$-distribution 
\begin{equation}
\label{eq: Steingrimsson}
\sum_{m \ge 0} \, (rm + 1)^n \, x^m \ = \ \frac{\sum_{w \in \fS_{n, r}
} \, x^{\des_{<_{\St}}(w)}}{(1 -x)^{n+1}},
\end{equation}
which reduces to \cref{eq: Euler} for $r=1$ and generalizes a formula of Brenti \cite[Theorem~3.4~(ii)]{Bre94} for $=2$.

Biagioli and Caselli \cite{BC12} studied a notion of descent set by considering the \emph{color~order} on $\Om_{n, r}$, defined as
\[
1^{r-1} <_c \cdots <_c n^{r-1} <_c \cdots <_c 1^1 <_c \cdots <_c n^1 <_c  1^0 <_c \cdots <_c n^0.
\]
For $w \in \fS_{n, r}$, we define $\Des_{<_c}(w)$ to be the set of all indices $i \in [n-1]$ such that  $w(i) >_c w(i+1)$ together with 0, whenever $c_1 \neq 0$ and write $\des_{<_c}(w)$ for its cardinality. It follows from their work \cite[Corollary~5.3~for~$p=s=q=1$]{BC12} that \cref{eq: Steingrimsson} holds if we replace $<_{\St}$ with $<_c$. 

Another notion of descent set was studied by Biagioli and Zeng \cite{BZ11}. In particular, let $<_{\ell}$ be the following total order on $\Om_{n, r}$ 
\[
n^{r-1} <_{\ell} \cdots <_{\ell} n^1 <_{\ell} \cdots <_{\ell} 1^{r-1} <_{\ell} \cdots <_{\ell} 1^1 <_{\ell} 1^0 <_{\ell} \cdots <_{\ell} n^0,
\]
called the \emph{length~order}. For $w \in \fS_{n, r}$, we define $\Des_{<_{\ell}}(w)$ to be the set of all indices $1 \le i \le n-1$ such that  $w(i) >_{\ell} w(i+1)$ together with 0, whenever $c_1 \neq 0$ and write $\des_{<_{\ell}}(w)$ for its cardinality. They proved \cite[Proposition~8.1~for~$q=1$]{BZ11} that \cref{eq: Steingrimsson} still holds if we replace $<_{\St}$ with $<_{\ell}$. It also follows from \cite[Proposition~7.1]{BC12}. So, the above three mentioned distributions are all equidistributed on $\fS_{n, r}$ and hence can be called Eulerian statistics on $r$-colored permutations. In the applications we will use the color order.

As a group, $\fS_{n, r}$ is generated by the set $S := \{s_0, s_1, \dots, s_{n-1}\}$, where $s_0 := \cycle{1^0, 1^1}$ and $s_i := \cycle{i^0, {i+1}^0}$, for all $1 \le i \le n-1$ (see, for example, \cite[Section~2]{AR01} and \cite[Section~2.2]{Bag04}).  The length function, written $\ell_S$, with respect to $S$ satisfies \cite[Theorem~4.4]{Bag04}
\begin{equation}
\label{eq: poincare series of Snr}
\sum_{w \in \fS_{n, r}} \, q^{\ell_S(w)} \ = \ [n]_q! \prod_{i =1}^n \, (1 + q^i[r-1]_q).
\end{equation}
From this point of view, a Mahonian statistic on $\fS_{n, r}$, is expected to be equidistributed with the length function. Bagno \cite[Theorem~5.2]{Bag04} introduced such a statistic by using the length order. It is defined as follows. For $w = w(1)^{c_1}w(2)^{c_2}\cdots w(n)^{c_n} \in \fS_{n, r}$, let
\[
\lmaj (w) := \maj_{<_\ell}(w) + \sum_{c_i \neq 0} (w(i) - 1) + \csum(w),
\]
where $\maj_{<_\ell} (w)$ is the sum of all elements of $\Des_{<_\ell} (w)$ and $\csum(w) = \sum_{i=1}^{n} \, c_i$ is the sum of the colors of all entries of $w$. It is worth noticing, as the authors in \cite[Section~7]{BC12} point out, that the length order seems to be the suitable order for proving a combinatorial interpretation of the length function of $\fS_{n, r}$, whereas the color order is often used in the study of some algebraic aspects such as the invariant theory of $\fS_{n, r}$.

Another Mahonian candidate, the flag major index, was introduced by Adin and Roichman \cite{AR01}. We use the following combinatorial interpretation \cite[Theorem~3.1]{AR01} as our definition. For $w \in \fS_{n, r}$, 
the \emph{flag~major~index} of $w$ is defined by
\[
\fmaj_{<_c}(w) := r\maj_{<_c}(w) + \csum(w),
\]
where $\maj_{<_c} (w)$ is the sum of all elements of $\Des_{<_c} (w)$. The authors remark, after the proof of \cite[Theorem~2.2]{AR01}, that the flag major index is not equidistributed with the length function on $\fS_{n, r}$ for $r \ge 3$. Haglund et. al were the first to explicitly compute \cite[Theorem~4.5]{HLR05} a formula for the $\fmaj$-distribution
\begin{equation}
\label{eq: fmaj distribution}
\sum_{w \in \fS_{n, r}} \, q^{\fmaj_{<_c}(w)} \ = \ \prod_{i =1}^n \, [ri]_q.
\end{equation}
Although the right-hand side does not coincide with that of \cref{eq: poincare series of Snr} for $r \ge 3$, \cref{eq: fmaj distribution} reduces to \cref{eq: MacMahon} for $r=1$ and the flag major index is a valid Mahonian statistic on $\fS_{n, r}$, in the following sense. It is known that $r, 2r, \dots, nr$ are the degrees of $\fS_{n, r}$, when viewed as a complex reflection group, so the right-hand side of \cref{eq: fmaj distribution} is the Hilbert series for the coinvariant algebra of $\fS_{n, r}$ (see \cite[Equation~(1.4)]{BRS08}).

Chow and Mansour \cite[Theorem~5]{CM11} prove a different interpretation of the flag major index using Steingr\'{i}msson's total order on $\Om_{n, r}$, namely
\[
\fmaj_{<_c}(w) \ = \ r \maj_{<_{\St}}(w) - \csum(w)
\]
for every $w \in \fS_{n, r}$. We denote by $\fmaj_{<_{\St}}$ the right-hand side of the above equation. It is also true that \cref{eq: fmaj distribution} holds if we replace $<_c$ by $<_{\ell}$ as the authors remark in \cite[Propostion~7.1]{BC12} and \cite[Remark~2.2]{BZ11}. Thus, $\fmaj_<$ for all $< \in \{<_c, <_{\St}, <_\ell\}$ on $\Om_{n, r}$ can be called Mahonian statistics on $r$-colored permutations. In the applications we will use the color order.

Now that we have explained what it means for a colored permutation statistic to be Eulerian and Mahonian, we discuss Euler--Mahonian pairs of statistics. The problem of extending the notion of Euler--Mahonian distribution to colored permutation groups was first proposed by Foata \cite[Problem~1.1]{ABR01} for the case $r=2$, of the hyperoctahedral group $\fB_n$.  Biagioli and Caselli \cite[Corollary~5.3~for~$s = p = 1$]{BC12} prove, in the more general setting of projective reflection groups, the following Euler--Mahonian identity
\begin{equation}
\label{eq: (des, fmaj) color}
\sum_{m \ge 0} \, [rm + 1]_q^n \, x^m \ = \ \frac{\sum_{w \in \fS_{n, r}} \, x^{\des_{<_c}(w)} q^{\fmaj_{<_c}(w)}}{(x; q^r)_{n +1}}.
\end{equation}
Equation \eqref{eq: (des, fmaj) color} reduces to Equations \eqref{eq: Carlitz} and \eqref{eq: Steingrimsson} for $r=1$ and $q=1$, respectively and generalizes a formula of Chow and Gessel \cite[Theorem~3.7]{CG07} for $r=2$. Chow and Mansour \cite[Theorem~9~(iv)]{CM11} prove that the above identity holds if we replace $<_c$ with $<_{\St}$. It is worth mentioning that, in the case $r=2$, Chow--Mansour's identity was first noticed by Biagioli and Zeng \cite[Section~3]{BZ10}. In a subsequent paper, the latter showed \cite[Proposition~8.1]{BZ11} that it also holds if we replace $<_c$ by $<_{\ell}$. Furthermore, they prove \cite[Equation~(8.1)~for~$a = 1$]{BZ11} the following identity 
\begin{equation}
\label{eq: (des, maj) length}
\sum_{m \ge 0} \, ([m+1]_q + (r-1)[m]_q)^n \, x^m \ = \ \frac{\sum_{w \in \fS_{n, r}} \, x^{\des_{<_{\ell}}(w)} q^{\maj_{<_{\ell}}(w)}}{(x; q)_{n +1}}
\end{equation}
which reduces to Equations \eqref{eq: Carlitz} and \eqref{eq: Steingrimsson} for $r=1$ and $q =1$, respectively and generalizes a formula of Chow and Gessel \cite[Equation~(26)]{CG07} for $r=2$.

For a colored permutation $w$ and a total order on $\Om_{n, r}$, let $\Des_<^*(w) := \{i \in [n-1] : w(i) > w(i+1)\}$ and write $\des_<^*(w)$ for its cardinality. The $*$-descent set is often called type $A$ descent set because it does not take into account the descents in positions $0$ or $n$ (see, for example, \cite[Definition~5.4]{BB13}. Bagno and Biagioli \cite{BB07} defined the  \emph{flag~descent~number}
\[
\fdes_{<_c}(w) \ := \ r \des_{<_c}^*(w) + c_1
\]
of a colored permutation  $w = w(1)^{c_1}w(2)^{c_2}\cdots w(n)^{c_n} \in \fS_{n, r}$, generalizing a notion first introduced by Adin, Brenti and Roichman \cite[Section~4]{ABR01} for the hyperoctahedral group $\fB_n$. They proved \cite[Theorem~A.~1]{BB07} that
\begin{equation}
\label{eq: (fdes, fmaj) color}
\sum_{m \ge 0} \, [m+1]_q^n \, x^m \ = \ \frac{\sum_{w \in \fS_{n, r}} \, x^{\fdes_{<_c}(w)} q^{\fmaj_{<_c}(w)}}{(1-x)(1 - x^r q^r)(1 - x^r q^{2r}) \cdots (1 - x^r q^{nr})},
\end{equation}
which generalizes Adin, Brenti and Roichman's formula \cite[Theorem~4.2]{ABR01} for $r=2$.

For the sake of completeness, we remark that Bagno \cite{Bag04} introduced an Eulerian partner to his $\lmaj$ statistic, defined by
\[
\ldes(w) \ := \ \des_{<_{\ell}}^*(w) + \csum(w),
\]
 for every $w \in \fS_{n, r}$ and proved the following Euler--Mahonian identity
\begin{equation}
\label{eq: bagno}
\sum_{m \ge 0} \, [m+1]_q^n \, x^m \ = \ \frac{\sum_{w \in \fS_{n, r}} \, x^{\ldes(w)} q^{\lmaj(w)}}{(x; q)_{n +1} (-x[r-1]_{qx}; q)_{n +1}}.
\end{equation}
Furthermore, Bagno introduced an Euler--Mahonian pair $(\ndes, \nmaj)$ on colored permutations, which satisfies Equation \eqref{eq: (fdes, fmaj) color}. This serves as a generalization of the $\text{\textquotedblleft{negative\textquotedblright}}$ statistics first considered by Adin, Brenti and Roichman \cite[Section~3]{ABR01}  for the hyperoctahedral group $\fB_n$.

%
%
\subsection{Colored quasisymmetric functions}
\label{subsec: quasi}

	The colored analogue of the fundamental quasisymmetric functions that we are going to use was introduced by Poirier \cite[Section~3]{Poi98} and further studied in \cite{BaH08,BeH06,HP10}. Different type $B$ generalizations of quasisymmetric functions have been suggested. See, for example, \cite[Section~2]{Pet05} for a brief account of Chow's type $B$ quasisymmetric functions.  Poirier's signed quasisymmetric functions were recently employed by Adin et al. \cite{AAER17} in order to provide a signed analogue of the concept of fine characters and fine sets of Adin and Roichman \cite{AR15} to the hyperoctahedral group $\fB_n$. This fact provided much of the motivation behind this paper in choosing  Porier's  fundamental colored quasisymmetric function to specialize.

	Let $\bx^{(j)} = (x_1^{(j)}, x_2^{(j)}, \dots)$, for every $0 \le j \le r-1$, be sequences of commuting indeterminates. We consider formal power series in $\bX^{(r)} := (x_i^{(0)}, x_i^{(1)}, \dots, x_i^{(r-1)})_{i \ge 1}$. The \emph{fundamental~colored~quasisymmetric~function} associated to $w = w(1)^{c_1}w(2)^{c_2}\cdots w(n)^{c_n} \in \fS_{n, r}$, written $F_{w, <}(\bX^{(r)})$,  is defined as follows
\begin{equation}
\label{eq: fundamental colored quasi}
F_{w, <} (\bX^{(r)}) \ := \ 
\sum_{\substack{i_1 \ge i_2 \ge \cdots \ge i_n \ge 1 \\ j \in \Des_<^*(w) \ \toto \  i_j > i_{j+1}}} \, x_{i_1}^{(c_1)} x_{i_2}^{(c_2)} \cdots x_{i_n}^{(c_n)},
\end{equation}
for any total order $<$ on $\Om_{n, r}$. This definition is slightly different, but equivalent to, the one given in \cite[Section~3.2]{HP10} with the inequalities being reversed for $< = <_c$, the color order (see also \cite[Equation~(2.9)]{AAER17} for the case $r=2$).

We extend this definition to standard Young $r$-partite tableaux, for which we mostly follow the exposition of \cite{BB07}. An \emph{$r$-partite~partition} of a nonnegative integer $n$ is any $r$-tuple $\bl = (\lambda^{(0)}, \dots, \lambda^{(r-1)})$ of (possibly empty) integer partitions of total sum $n$. We write $\bl \vdash n$ for $r$-partite partitions. A \emph{standard~Young~$r$-partite~tableau} of shape $\bl = (\lambda^{(0)}, \dots, \lambda^{(r-1)}) \vdash n$ is an $r$-tuple $\bQ = ( Q^{(0)}, \dots, Q^{(r-1)})$ of tableaux which are strictly increasing along rows and columns such that $Q^{(i)}$ has shape $\lambda^{(i)}$, for all $0 \le i \le r-1$ and every element of $[n]$ appears exactly once as an entry of $Q^{(i)}$ for some $0 \le i \le r-1$. These tableaux are called \emph{parts} of $\bQ$. Let $\SYT(\bl)$ (resp. $\SYT_{n, r}$) be the set of all standard Young $r$-partite tableaux of shape $\bl \vdash n$ (resp. of any shape and size $n$).

For $\bQ = ( Q^{(0)}, \dots, Q^{(r-1)}) \in \SYT_{n, r}$, an integer $i \in [0, n-1]$ is called a \emph{descent} of $\bQ$, if
\begin{itemize}
\item $i$ and $i+1$ belong in the same part of $\bQ$ and $i+1$ appears in a lower row than $i$ does, or
\item $i \in Q^{(j)}$ and $i +1 \in Q^{(k)}$, for some $0 \le j < k \le r-1$, or
\item  $i = 0$ and $1$ appears in $Q^{(j)}$ for some $j \neq 0$.
\end{itemize}
The set of all descents of $\bQ$, written $\Des(\bQ)$, is called the \emph{descent set} of $\bQ$. The cardinality and the sum of all elements of $\Des(\bQ)$ are written $\des(\bQ)$ and $\maj(\bQ)$, respectively and called the \emph{descent number} and \emph{major index} of $\bQ$. Also, let $\Des^*(\bQ)$ be the set obtained from $\Des(\bQ)$ by removing the zero, if present.

We recall that the Robinson--Schensted correspndence is a bijection from the symmetric group $\fS_n$ to the set of pairs $(P, Q)$ of standard Young tableaux of the same shape and size $n$. It has the property \cite[Lemma~7.23.1]{StaEC2} that $\Des(w) = \Des(Q)$ and $\Des(w^{-1}) = \Des(P)$, where $(P, Q)$ is the pair of tableaux associated to $w \in \fS_n$. The Robinson--Schensted correspondence has a natural colored analogue, first considered by White \cite[Corollary~9~and~Remark~11]{Whi83} and further studied by Stanton and White \cite{SW85}. It is a bijection from the $r$-colored permutation group $\fS_{n, r}$ to the set of pairs $(\bP, \bQ)$ of standard Young $r$-partite tableaux of the same shape and size $n$. It has the property \cite[Proposition~6.2]{BRS08}, \cite[Lemma~5.2]{APR10} that $\Des_{<_c}(w) = \Des(\bQ)$ and $\Des_{<_c}(\ol{w}^{-1}) = \Des(\bP)$, where $(\bP, \bQ)$ is the pair of $r$-partite tableaux associated to $w \in \fS_{n, r}$ and  $\ol{w}$ is the colored permutation with underlying permutation $w$ and color vector $(-c_1, -c_2, \dots, -c_n)$, where the entries are computed modulo $r$. For the case $r=2$, we refer the reader to \cite[Section~5]{AAER17} and references therein.

The fundamental colored quasisymmetric function associated to a tableau $\bQ$ is then defined by the right-hand side of \cref{eq: fundamental colored quasi}, with $w$ replaced by $\bQ$ and $\Des_<^*(w)$ replaced by $\Des^*(\bQ)$. A signed analogue of the following well known expansion \cite[Theorem~7.19.7]{StaEC2}
\begin{equation}
\label{eq: schurtofun}
s_{\lambda} (\bx) = \sum_{Q \in \SYT(\lambda)} \, F_{n, \Des(Q)} (\bx)
\end{equation} 
was recently proved by Adin et al. \cite[Proposition~4.2]{AAER17}. The following colored analogue 
\begin{equation}
\label{eq: schurtocoloredfun}
s_\bl (\bX^{(r)}) \ = \ \sum_{\bQ \in \SYT(\bl)} \, F_{\bQ} (\bX^{(r)})
\end{equation}
of \cref{eq: schurtofun}, where 
\[
s_\bl (\bX^{(r)}) := s_{\lambda^{(0)}}(\bx^{(0)}) s_{\lambda^{(1)}}(\bx^{(1)}) \cdots s_{\lambda^{(r-1)}}(\bx^{(r-1)}),
\]
for every $r$-partite partition $\bl = (\lambda^{(0)}, \lambda^{(1)}, \dots, \lambda^{(r-1)})$ follows from a trivial generalization of the proof of \cite[Proposition~4.2]{AAER17} and will be used in the sequel.

%
%
\subsection{Representation theory; the characteristic map}
\label{subsec: reptheory}

	We denote by $\Lambda(\bx)$ the $\CC$-algebra of symmetric functions in $\bx$. For the representation theory of the symmetric group and its connection to symmetric functions, we refer to \cite[Section~7.18]{StaEC2} and only recall that the \emph{Frobenius~characteristic~map}, written $\ch$, is a $\CC$-linear isomorphism from the space of virtual $\fS_n$-representations to that of homogeneous symmetric functions of degree $n$, with the property that $\ch (\chi^{\lambda}) (\bx) = s_{\lambda}(\bx)$, where $\chi^{\lambda}$ is the irreducible $\fS_n$-character corresponding to $\lambda \vdash n$. The map $\ch$ has a natural colored analogue \cite[Appendix~B]{MacSFHP}, which we now describe. We mostly follow the exposition of \cite[Sections~5~and~6]{BB07} and \cite[Section~2]{Poi98}.

	Conjugacy class in $\fS_{n, r}$, and therefore irreducible $\fS_{n, r}$-characters, are in one-to-one correspondence with $r$-partite partitions of $n$. We describe the cycle type of a colored permutation. Starting with a colored permutation, we first form the cycle decomposition of the underlying permutation and then provide the entries with their original color, forming \emph{colored~cycles}. Then, define the color of a colored cycle to be the sum of colors of all its entries computed modulo $r$. Now, the cycle type of $w \in \fS_{n, r}$, written $\ct(w)$, is the $r$-partite partition of $n$, whose $j$-th part is the integer partition formed by the lengths of the colored cycles of  $w$ having color $j$, for every $0 \le j \le r-1$. Two colored permutations are conjugate \cite[Proposition~1]{Poi98} if and only if they have the same cycle type.

	Fix $\zeta$ a primitive $r$-th root of unity. For a nonnegative integer $k$ and any $0 \le j \le r-1$, let
\[
p_k^{(j)} (\bX^{(r)}) \ := \ p_k (\bx^{(0)}) + \zeta^j p_k (\bx^{(1)}) + \cdots + \zeta^{j(r-1)} p_k (\bx^{(r-1)}).
\]
Also, for an integer partition $\lambda = (\lambda_1, \lambda_2, \dots) \vdash n$ and any $0 \le j \le r-1$, define the homogeneous element 
\[
p_\lambda^{(j)} (\bX^{(r)}) \ := \ p_{\lambda_1}^{(j)} (\bX^{(r)}) p_{\lambda_2}^{(j)} (\bX^{(r)})\cdots 
\]
of degree $n$ of $\Lambda(\bx^{(0)}) \otimes \Lambda(\bx^{(1)}) \otimes \cdots \otimes \Lambda(\bx^{(r-1)})$. The \emph{colored~power~sum~symmetric~function} associated to a $r$-partite partition $\bl = (\lambda^{(0)}, \dots, \lambda^{(r-1)}) \vdash n$, is defined by
\[
p_{\bl} (\bX^{(r)}) \ := \  p_{\lambda^{(0)}}^{(0)} (\bX^{(r)}) p_{\lambda^{(1)}}^{(1)} (\bX^{(r)}) \cdots p_{\lambda^{(r-1)}}^{(r-1)} (\bX^{(r)}). 
\]

	The \emph{Frobenius~characteristic~map} of a complex, finite-dimensional $\fS_{n, r}$-character $\chi$ is defined by 
\begin{equation}
\label{eq: FrobDef}
\ch_r (\chi) (\bX^{(r)}) \ := \ \frac{1}{r^nn!} \sum_{w \in \fS_{n, r}} \, \chi(w) p_{\ct(w)} (\bX^{(r)}). 
\end{equation}
The map $\ch_r$ is a $\CC$-linear isomorphism from the space of virtual $\fS_{n, r}$-characters to the degree $n$ homogeneous part of $\Lambda(\bx^{(0)}) \otimes \Lambda(\bx^{(1)}) \otimes \cdots \otimes \Lambda(\bx^{(r-1)})$. It has the property that $\ch_r (\chi^{\bl}) (\bX^{(r)}) \ = \ s_\bl (\bX^{(r)})$, where $\chi^{\bl}$ is the irreducible $\fS_{n, r}$-character associated to $\bl \vdash n$.

	For a subset $\cC \subseteq \fS_{n, r}$ and a total order $<$ on $\Om_{n, r}$,  let 
\[
F_< (\cC; \bX^{(r)}) \ := \ \sum_{w \in \cC} \, F_{w, <}(\bX^{(r)})
\]
be the colored quasisymmetric generating function associated to $\cC$. The following observation, which appears in \cite[Proposition~1.13]{PoiThesis} in a more general setting, is a generalization of Equation \eqref{eq: F(Sn,x)} and will be used in \cref{sec: applications}. It is the key that allows us to pass from general formulas to Euler--Mahonian identities by appropriately specializing it. We record it here with a proof. 
\begin{lemma}
\label{lem: coloredhelp}
For a nonnegative integer $n$, we have
\begin{equation}
\label{eq: F(Snr)}
F_{<_c} (\fS_{n, r}; \bX^{(r)}) \ = \ (x_1^{(0)} + x_2^{(0)} + \cdots + x_1^{(r-1)} + x_2^{(r-1)} + \cdots)^n.
\end{equation}
\end{lemma}
\begin{proof}
	Recall from \cite[Theorem~6.1]{BB07}, the Frobenius formula 
\begin{equation}
\label{eq: colorFrobeniusformula}
p_\bm (\bX^{(r)})\ = \ \sum_{\bl \vdash n} \, \chi^{\bl} (\bm) s_\bl (\bX^{(r)}), 
\end{equation}
for $\fS_{n, r}$, where $\chi^{\bl} (\bm)$ is the irreducible $\fS_{n, r}$-character associated to $\bl$ computed in the conjugacy class which corresponds to $\bm \vdash n$. Notice that the conjugacy class of the identity element in $\fS_{n, r}$ corresponds to the $r$-partite partition of $n$, whose part of color $0$ is $(1^n)$ and all the other parts are the empty partitions, written $(1^n, \varnothing^{r-1})$, and so the left-hand side of \cref{eq: colorFrobeniusformula} is equal to
\[
p_{(1^n, \varnothing^{r-1})} (\bX^{(r)}) \ = \ p_{(1^n)}^{(0)} (\bX^{(r)}) \ = \ (p_1^{(0)} (\bX^{(r)}) )^n \ = \ (p_1 (\bx^{(0)}) + \cdots + p_1 (\bx^{(r-1)}))^n,
\]
which is exactly the right-hand side of Equation \eqref{eq: F(Snr)}. Furthermore, $\chi^{\bl} (1^n, \varnothing^{r-1}) = \dim_{\CC} (\chi^{\bl})$, which is known to equal the number of $r$-partite standard Young tableaux of shape $\bl$. Therefore, \cref{eq: colorFrobeniusformula} for $\bm = (1^n, \varnothing^{r-1})$ becomes
\[
p_{(1^n, \varnothing^{r-1})} (\bX^{(r)}) = \sum_{\bl \vdash n} \, \sum_{\bP, \bQ \in \SYT(\bl)} \, F_{\bQ} (\bX^{(r)}),
\]
using the expansion \eqref{eq: schurtocoloredfun}. This in turn is exactly the left-hand side of Equation \eqref{eq: F(Snr)} by the colored Robinson--Schensted correspondence and its properties. The proof follows by combining the two calculations.
\end{proof} 
%

%
%
\section{Specializations of colored quasisymmetric functions}
\label{sec: special}

	This section derives general formulas for Mahonian and Euler-Mahonian distributions on colored permutation groups by specializing colored quasisymmetric functions. Let $\CC[[\bX^{(r)}]]$ be the ring of formal power series in $\bX^{(r)}$. Formally, a \emph{specialization} is a ring homomorphism $\CC[[\bX^{(r)}]] \to \CC[[q]]$ or $\CC[[\bX^{(r)}]] \to \CC[q]$. In this paper, we consider specializations that arise from substituting powers of $q$ for the variables $x_i^{(j)}$, for $0 \le j \le r-1$. 

	Fix a total order $<$ on $\Om_{n, r}$ and define the descent set of a colored permutation $w = w(1)^{c_1}w(2)^{c_2}\cdots w(n)^{c_n} \in \fS_{n, r}$, written $\Des_<(w)$, to be the set of all indices $i \in [n-1]$, such that $w(i) > w(i+1)$, together with $0$, whenever $c_1 \neq 0$, and $\des_<(w)$ its cardinality. Let $\Des_<^*(w)$ be the set obtained from $\Des_<(w)$ by removing zero, if present, and let $\des_<^*(w)$ be its cardinality. Also, define the major index of $w$, written $\maj_<(w)$, to be the sum of all elements of $\Des_<(w)$ and the flag major index of $w$ as $\fmaj_<(w) := r\maj_<(w) + \csum(w)$. The results of this section are valid for every total order on $\Om_{n, r}$ and therefore we omit the total order subscript in colored statistics and colored quasisymmetric functions. 

	We begin with two specializations and a variation which reduce to the stable specialization and the principal specialization of order $m$ for $r=1$. Let $\ps^{(r)}$ be the specialization defined by the substitutions $x_i^{(j)} = q^{i -1}$, for every $0 \le j \le r-1$ and $i \ge 1$. Also, define the specialization $\ps_m^{(r)}$ by the substitutions $x_i^{(0)} = q^{i - 1}$, for every $1 \le i \le m$ and $x_i^{(j)} = q^{i -1}$, for every $1 \le i \le m-1$ and $1 \le j \le r-1$  and zero elsewhere. In fact, $\ps^{(r)}$ is the stable principal specialization on all $\bx^{(0)}, \bx^{(1)}, \dots, \bx^{(r-1)}$ and $\ps_m$ is the principal specialization of order $m$ on $\bx^{(0)}$ and of order $m-1$ on $\bx^{(1)}, \bx^{(2)}, \dots, \bx^{(r-1)}$. Furthermore, let $\widetilde{\ps}_m^{(r)}$ be the principal specialization of order $m$ on all $\bx^{(0)}, \bx^{(1)}, \dots, \bx^{(r-1)}$. 
\begin{theorem}
\label{thm: Icolor}
For a positive integer $n$ and every $w \in \fS_{n, r}$, we have 
\begin{align}
\ps^{(r)} (F_w) \ &= \ \frac{q^{\maj(w)}}{(q)_n} \label{eq: Icolormaj}\\ 
\sum_{m \ge 1} \, \ps_m^{(r)} (F_w) \, x^{m-1} \ &= \ \frac{x^{\des(w)} q^{\maj(w)}}{(x; q)_{n +1}} \label{eq: Icolordesmaj}\\
\sum_{m \ge 1} \, \widetilde{\ps}_m^{(r)} (F_w) \, x^{m-1} \ &= \ \frac{x^{\des^*(w)} q^{\maj(w)}}{(x; q)_{n +1}}. \label{eq: Icolordes*maj}
\end{align}
\end{theorem}
\begin{proof}
	We prove Equations \eqref{eq: Icolormaj} and \eqref{eq: Icolordesmaj} in parallel. Equation \eqref{eq: Icolordes*maj} follows in a similar way. For a colored permutation $w$ with color vector $(c_1, c_2, \dots, c_n)$, we have
\begin{align}
 \ps^{(r)} (F_w)  \ &= \ \sum_{\substack{i_1 \ge i_2 \ge \cdots \ge i_n \ge 1 \\ j \in \Des^*(w) \ \toto \ i_j > i_{j+1}}}  \, q^{i_1 + i_2 + \cdots + i_n - n} \label{eq: Ia}\\
 \ps_m^{(r)} (F_w) \ &= \ \sum_{\substack{m := i_0 \ge i_1 \ge i_2 \ge \cdots \ge i_n \ge 1 \\ j \in \Des(w) \ \toto \ i_j > i_{j+1}}} q^{i_1 + i_2 + \cdots + i_n - n}. \label{eq: Ib}
 \end{align}
Under the specialization $\ps_m^{(r)}$, substitutions $x_m^{(1)},$ $x_m^{(2)}, \dots, x_m^{(r-1)}$ occur only if $c_1 \neq 0$, which in turn is exactly when 0 is considered a descent of $w$, explaining the first inequality under the sum on the right-hand side of \cref{eq: Ib}. Define 
\begin{align*}
i_j' \ &= \ i_j - \chi_j - \cdots - \chi_{n-1} \\
i_n' \ &= \ i_n,
\end{align*}
where $\chi_j := \chi (j \in \Des(w))$, for every $0 \le j \le n-1$.  Then, Equations \eqref{eq: Ia} and \eqref{eq: Ib} become
\begin{align}
 \ps^{(r)} (F_w)  \ &= \ \sum_{i_1' \ge i_2' \ge \cdots \ge i_n' \ge 1} \, q^{i_1' + i_2' + \cdots + i_n' - n + \maj(w)} \label{eq: Iaa} \\
 \ps_m^{(r)} (F_w)  \ &= \ \sum_{m - \des(w) \ge i_1' \ge i_2' \ge \cdots \ge i_n' \ge 1} \, q^{i_1' + i_2' + \cdots + i_n' - n + \maj(w)}, \label{eq: Ibb} 
\end{align}
because 
\begin{align*}
\des (w) \ &= \ \sum_{j = 0}^{n-1} \, \chi_j \\
\maj(w) \ &= \ \sum_{j = 1}^{n-1} \, j \chi_j. 
\end{align*}
Now, let
\begin{align*}
a_0 \ &= \ m - \des (w) - i_1' \\
a_j \ &= \ i_j' - i_{j+1}' \\
a_n \ &= \ i_n' - 1,
\end{align*}
for every $1 \le j \le n-1$. On the one hand, \cref{eq: Iaa} becomes 
\[
\ps^{(r)} \ = \ \sum_{a_1, a_2, \dots, a_n \in \NN} \, q^{a_1 + 2a_2 + \cdots + na_n + \maj(w)} \ = \ \frac{q^{\maj(w)}}{(q)_n}. 
\]
On the other hand, \cref{eq: Ibb} becomes 
\[
\ps_m^{(r)} (F_w) \ = \ 
\sum \,  q^{a_1 + 2a_2 + \cdots + na_n + \maj(w)}
\]
where the sum runs through all $\NN$-solutions of $a_0 + a_1 + \cdots + a_n = m - \des(w) - 1$. This is exactly the coefficient of $x^{m-1}$ in the expansion of the right-hand side of \cref{eq: Icolordesmaj} and the proof follows. 
\end{proof}

	Notice that in the proof of \cite[Lemma~5.2]{GR93} (see also \cite[Lemma~7.19.10]{StaEC2}) the authors deal with the comajor index instead of the major index. Our choice of the direction of inequalities in the definition of fundamental colored quasisymmetric functions (see \cref{eq: fundamental colored quasi}) allows us to deal with the major index directly. This observation explains the motivation behind our choice. The following formulas are immediate consequences of Theorem \ref{thm: Icolor}.
\begin{corollary}
\label{cor: Icolor}
For a positive integer $n$ and every $\cC \subseteq \fS_{n, r}$, we have 
\begin{align}
\ps^{(r)} (F (\cC)) \ &= \ \frac{\sum_{w \in \cC}  \, q^{\maj(w)}}{(q)_n} \label{eq: IcolormajF}\\ 
\sum_{m \ge 1} \, \ps_m^{(r)} (F(\cC)) \, x^{m-1} \ &= \ \frac{\sum_{w \in \cC} \, x^{\des(w)} q^{\maj(w)}}{(x; q)_{n +1}} \label{eq: IcolordesmajF}\\
\sum_{m \ge 1} \, \widetilde{\ps}_m^{(r)} (F(\cC)) \, x^{m-1} \ &= \ \frac{\sum_{w \in \cC} \, x^{\des^*(w)} q^{\maj(w)}}{(x; q)_{n +1}}. \label{eq: Icolordes*majF}
\end{align}
\end{corollary} 

	Although these formulas do not provide much useful information by themselves, in \cref{sec: applications} we apply \cref{cor: Icolor} (and corollaries that follow) for various $\cC$, whose quasisymmetric generating function $F(\cC; \bX^{(r)})$ has a nice form to prove Euler--Mahonian identities on $\cC$. Thus, whenever one proves a formula for the quasisymmetric generating function of a collection $\cC$ of colored permutations (for example, in the case where $F(\cC; \bX^{(r)})$ is symmetric and Schur-positive in $\Lambda(\bx^{(0)}) \otimes \Lambda(\bx^{(1)}) \otimes \cdots \otimes \Lambda(\bx^{(r-1)})$), which can be specialized in a nice way, then an Euler--Mahonian (resp. Mahonian) identity is automatically obtained from \cref{eq: IcolordesmajF} (resp. \cref{eq: IcolormajF}). 
	
	Next, we consider another two specializations and a variation which can be viewed as colored-shifted versions of the stable principal specialization and the principal specialization of order $m$ for $r=1$. Let $\psi^{(r)}$ be the  specialization defined by the substitutions $x_i^{(j)} = q^{r(i - 1) + j}$, for every $0 \le j \le r-1$ and $i \ge 1$. Also, define the specialization $\psi_m^{(r)}$ by the substitutions	
\[
x_i^{(0)} \ = \ q^{r (i - 1)},
\]
for every $1 \le i \le m$ and 
\[
x_i^{(j)} \ = \ q^{r (i - 1) + j},
\]
for every $1 \le i \le m-1$ and $1 \leq j \le r-1$ and zero elsewhere. Furthermore, let $\widetilde{\psi}_m^{(r)}$ be the specialization defined as $\psi_m^{(r)}$, but including the substitution $x_m^{(j)} \ = \ q^{r (m-1) + j}$, for every $1 \le j \le r-1$. 
\begin{theorem}
\label{thm: IIcolor}
For a positive integer $n$ and every $w \in \fS_{n, r}$, we have
\begin{align}
\psi^{(r)} ( F_w ) \ &= \ \frac{q^{\fmaj(w)}}{(q^r)_n} \label{eq: IIcolorfmaj} \\
\sum_{m \ge 1} \, \psi_m^{(r)} (F_w) \, x^{m-1} \ &= \ \frac{x^{\des(w)} q^{\fmaj(w)}}{(x; q^r)_{n +1}} \label{eq: IIcolor(des, fmaj)} \\
\sum_{m \ge 0} \, \widetilde{\psi}_m^{(r)} (F_w) \, x^{m-1} \ &= \  \frac{x^{\des^*(w)} q^{\fmaj(w)}}{(x; q^r)_{n +1}}. \label{eq: IIcolor(des*, fmaj)}
\end{align}
\end{theorem}

	The proof of \cref{thm: IIcolor} is a slight variation of the proof of \cref{thm: Icolor} and is therefore omitted. The following formulas are immediate consequences of \cref{thm: IIcolor}.
\begin{corollary}
\label{cor: IIcolor}
For a positive integer $n$ and every $\cC \subseteq \fS_{n, r}$, we have 
\begin{align}
\psi^{(r)} (F (\cC)) \ &= \ \frac{\sum_{w \in \cC}  \, q^{\fmaj(w)}}{(q^r)_n} \label{eq: IIcolorfmajF}\\ 
\sum_{m \ge 1} \, \psi_m^{(r)} (F(\cC)) \, x^{m-1}  \ &= \ \frac{\sum_{w \in \cC} \, x^{\des(w)} q^{\fmaj(w)}}{(x; q^r)_{n +1}} \label{eq: IIcolordesfmajF}\\
\sum_{m \ge 1} \, \widetilde{\psi}_m^{(r)} (F(\cC)) \, x^{m-1} \ &= \ \frac{\sum_{w \in \cC} \, x^{\des^*(w)} q^{\fmaj(w)}}{(x; q^r)_{n +1}}. \label{eq: IIcolordes*fmajF}
\end{align}
\end{corollary}

	The description of the next specialization is slightly more complicated than the previous ones and for this reason we consider the case $r=2$ first. In this case, we write $\bx$ and $\by$ instead of $\bx^{(0)}$ and $\bx^{(1)}$. Let $\phi_m$ be the specialization defined by the substitutions $x_i = q^{i-1}$, for every odd integer $1 \le i \le m$ and $y_i = q^i$, for every odd integer $1 \le i \le m-1$ and $x_i = y_i = 0$, elsewhere. In particular,
\[
x_m \ = \ 
	\begin{cases}
		q^{m-1}, \ &\text{if $m \equiv 1 \pmod{2}$} \\
		0, \ &\text{if $m \equiv 0 \pmod{2}$}
	\end{cases}
\]
and
\[
y_{m-1} \ = \ 
	\begin{cases}
		0, \ &\text{if $m \equiv 1 \pmod{2}$} \\
		q^{m-1}, \ &\text{if $m \equiv 0 \pmod{2}$}.
	\end{cases}
\]

	Now, in the  general case define the specialization $\phi_m^{(r)}$ by the substitutions
\[
x_i^{(j)} \ = \ 
	\begin{cases}
  		q^{i - 1 + j}, & \, \text{for every $1 \le i \le m$ with $i \equiv 1 \pmod{r}$} \\
		0, & \, \text{otherwise}
	\end{cases}
\]
for every $0 \le j \le r-1$, such that 
\begin{itemize}
\item $x_m^{(1)} = x_m^{(2)} = \cdots = x_m^{(r-1)} = 0$, and
\item if $m \equiv c \pmod{r}$, then the last non zero substitution is
\[
x_{m - c + 1}^{(c-1)} \ = \ q^{m-1},
\]
where $1 \le c \le r$. 
\end{itemize}
In particular, we see that if $m \equiv 1 \pmod{r}$, then the last nonzero term is $x_m^{(0)}$, if $m \equiv 2 \pmod{r}$, then the last non zero term is $x_{m - 1}^{(1)}$, and so on. Thus, $\phi^{(r)}_m$ coincides with $\phi_m$ for $r=2$.  
\begin{theorem}
\label{thm: IIIcolor}
For a positive integer $n$ and every $w \in \fS_{n, r}$, we have 
\begin{equation}
\label{eq: IIIcolorfdesfmaj}
\sum_{m \ge 1} \, \phi_m^{(r)} ( F_w ) \,x^{m-1} \ = \ \frac{x^{\fdes(w)} q^{\fmaj(w)}}{(1-x)(1 - x^r q^r)(1 - x^r q^{2r}) \cdots (1 - x^r q^{nr})}.
\end{equation}
Furthermore, for every $\cC \subseteq \fS_{n, r}$ we have 
\begin{equation}
\label{eq: IIIcolorfdesfmajF}
\sum_{m \ge 1} \, \phi_m^{(r)} ( F(\cC )) \, x^{m-1} \ = \ \frac{\sum_{w \in \cC} \, x^{\fdes(w)} q^{\fmaj(w)}}{(1-x)(1 - x^r q^r)(1 - x^r q^{2r}) \cdots (1 - x^r q^{nr})}.
\end{equation}
\end{theorem}
\begin{proof}
For $w \in \fS_{n, r}$ with color vector $(c_1, c_2, \dots, c_n)$, specialization $\phi^{(r)}_m$, when applied to the fundamental colored  quasisymmetric function, becomes
\[
\phi_m^{(r)} (x_{i_j}^{(c_j)}) \ = \ q^{i_j - 1 + c_j},
\]
subject to some restrictions. Thus, we have
\begin{equation}
\label{eq: IIIa}
\phi_m^{(r)} ( F_w) \ = \ 
\sum_{\substack{m := i_0 \ge i_1 \ge i_2 \ge \cdots \ge i_n \ge 1 \\ j \in \Des(w) \, \toto \, i_j > i_{j+1} \\ i_1, \dots, i_n \equiv 1 \pmod{r}}} \, q^{i_1 + i_2 + \cdots + i_n - n + \csum(w)},
\end{equation}
because as in \cref{eq: Ib},  $x_m^{(1)}, x_m^{(2)}, \dots, x_m^{(r-1)}$ occur only if $c_1 \ne 0$, which means that $0$ is a descent of $w$. Define 
\begin{align*}
i_0' \ &=\  i_0 - c_1 - r\chi_1 - \cdots - r\chi_{n -1} \\
i_j' \ &= \ i_j - r\chi_j - \cdots - r\chi_{n - 1} \\
i_n' \ &= \ i_n,
\end{align*}
where $\chi_j := \chi (j \in \Des(w))$, for every $1 \le j \le n-1$. Then, 	  \cref{eq: IIIa} becomes
\begin{equation}
\label{eq: IIIaa}
\phi_m^{(r)} (F_w) \ = \ 
\sum_{\substack{m - \fdes(w) \ge i_1' \ge i_2' \ge \cdots \ge i_n' \ge 1 \\ i_1', \dots i_n' \ \equiv \ 1 \pmod{r}}} \,
q^{i_1' + i_2' + \cdots + i_n' - n + \fmaj(w)},
\end{equation}
because
\begin{align*}
\fdes (w) \ &= \ c_1 + r\sum_{j = 1}^{n-1} \, \chi_j \\
\fmaj (w) \ &= \ r\sum_{j = 1}^{n-1}\,  j\chi_j + \csum(w).
\end{align*}
The first inequality in \cref{eq: IIIaa} is justified by the fact that the last non zero substitution is $x_{m - c + 1}^{(c-1)} \ = \ q^{m-1}$, where $m \equiv c \pmod{r}$, for every $1 \le c \le r$. Now, making the substitution
\begin{align*}
a_0 \ &= \ m - \fdes(w) - i_1' \\
a_j \ &= \ i_j' - i_{j + 1}' \\
a_n \ &= \ i_n' - 1,
\end{align*}
for every $1 \le j \le n-1$, \cref{eq: IIIaa} becomes 
\begin{equation}
\label{eq: IIIaaa}
\phi_m^{(r)} (F_w) \ = \ \sum \, q^{a_1 + 2a_2 + \cdots + na_n + \fmaj(w)},
\end{equation}
where the sum runs through all $\NN$-solutions of $a_0 + a_1 + a_2 + \cdots + a_n = m - \fdes(w) -1$ with the requirement that $a_1, a_2, \dots, a_n \equiv 0 \pmod{r}$ , as the difference of two positive integers congruent to $1 \pmod{r}$. The right-hand side of \cref{eq: IIIaaa} is precisely the coefficient of $x^{m - 1}$ in the expansion of \cref{eq: IIIcolorfdesfmaj} and the proof follows.
\end{proof}

	Lastly, we introduce a two parameter flag major index for signed permutations. Recall from \cref{subsec: colored} that in the case $r=2$, we identify $1^{(0)}, 2^{(0)}, \dots, n^{(0)}$ and $1^{(1)}, 2^{(1)}, \dots, n^{(1)}$, with $1, 2, \dots, n$ and $\ol{1}, \ol{2}, \dots, \ol{n}$, respectively. For $w \in \fB_n$, we write $\neg(w)$ instead of $\csum(w)$ to denote the number of indices $i \in [n]$ such that $w(i)$ is barred. As in the case of general $r$, we fix a total order $<$ on $\Om_{n, 2}$ and define $\Des_<(w)$ to be the set of all indices $i \in [n-1]$, such that $w(i) > w(i+1)$, together with $0$ whenever $w(1)$ is barred. The cardinality and the sum of all elements of $\Des_<(w)$ are written as $\des_<(w)$ and $\maj_<(w)$, respectively. Let $\Des_<^*(w)$ be the set obtained from $\Des_<(w)$ by removing zero, if present, and let $\des_<^*(w)$ be its cardinality. Also, define the flag major index of $w$ to be $\fmaj_<(w) := 2\maj_<(w) + \neg(w)$. Again, the results that follow are valid for every total order on $\Om_{n, 2}$ and therefore we omit the total order subscript.
	
	Let $k$ and $\ell$ be a positive, and nonnegative respectively, integer. For $w \in \fB_n$, define
\[
\fmaj_{k, \ell} (w) := k\maj(w) + \ell\neg(w) 
\]
the \emph{$(k, \ell)$-flag-major~index} of $w$.  The $(1, 0)$-flag-major index coincides with the major index and the $(2, 1)$-flag-major index is just the flag major index on signed permutations. We are going to derive general formulas for the pair $(\des, \fmaj_{k, \ell})$, by considering a $(k, \ell)$-variation of the specializations of \cref{thm: IIcolor} for $r=2$. 

	Let $\theta$ be the specialization defined by substitutions $x_i \ = \ q^{k(i-1)}$ and $y_i \ = \ q^{k(i-1) + \ell}$, for every $i \ge 1$. Also, define the specialization $\theta_m$ by the substitutions $x_i \ = \ q^{k(i-1)}$, for every $1 \le i \le m$ and $y_i \ = \ q^{k(i-1) + \ell}$, for every $1 \le i \le m-1$. For $(k, \ell) = (1, 0)$ and 
$(k, \ell) = (2, 1)$, these specializations coincide with $\ps^{(2)}, \ps_m^{(2)}$ and $\psi^{(2)}, \psi_m^{(2)}$, respectively. Furthermore, let $\widetilde{\theta}_m$ be the specialization defined as $\theta_m$, but including the substitution $y_m = q^{k(m-1) + \ell}$.
\begin{theorem}
\label{thm: IVklvariation}
For a positive integer $n$ and every $w \in \fB_n$, we have 
\begin{align}
\theta(F_w) \ &= \ \frac{q^{\fmaj_{k, \ell}(w)}}{(q^k)_n} \label{eq: IVfmaj}\\ 
\sum_{m \ge 1} \, \theta_m(F_w) \, x^{m-1} \ &= \ \frac{x^{\des(w)} q^{\fmaj_{k, \ell}(w)}}{(x; q^k)_{n +1}} \label{eq: IVdesfmaj}\\
\sum_{m \ge 1} \, \widetilde{\theta}_m(F_w) \, x^{m-1}  \ &= \ \frac{x^{\des^*(w)} q^{\fmaj_{k, \ell}(w)}}{(x; q^k)_{n +1}}. \label{eq: IVdes*fmaj}
\end{align}
\end{theorem}

	The proof of \cref{thm: IVklvariation} is a $(k, \ell)$-variation of the proof of \cref{thm: Icolor} and is therefore omitted. The following formulas are immediate consequences of \cref{thm: IVklvariation}.
\begin{corollary}
\label{cor: IVklvariation}
For a positive integer $n$ and every $\bB \subseteq \fB_n$, we have 
\begin{align}
\theta(F(\bB))  \ &= \ \frac{\sum_{w \in \bB}  \, q^{\fmaj_{k, \ell}(w)}}{(q^k)_n} \label{eq: IVfmaj F}\\ 
\sum_{m \ge 1} \, \theta_m(F (\bB))  \, x^{m-1} \ &= \ \frac{\sum_{w \in \bB} \, x^{\des(w)} q^{\fmaj_{k, \ell}(w)}}{(x; q^k)_{n + 1}} \label{eq: IVdesfmajF}\\
\sum_{m \ge 1} \, \widetilde{\theta}_m(F (\bB)) \, x^{m-1} \ &= \ \frac{\sum_{w \in \bB} \, x^{\des^*(w)} q^{\fmaj_{k, \ell}(w)}}{(x; q^k)_{n +1}}. \label{eq: IVdes*fmajF}
\end{align}
\end{corollary} 
%

%
%
\section{Applications}
\label{sec: applications}

	This section applies the corollaries of \cref{sec: special} to prove Euler--Mahonian identities on colored permutation groups. In particular, Section \ref{subsec: Euler--Mahonian identities} proves most of the Euler--Mahonian identities mentioned in \cref{subsec: colored} and introduces several new examples. \Cref{subsec: derangem} studies Mahonian and Euler--Mahonian distributions on derangements and their colored analogues and proves an Euler--Mahonian identity on derangements, refining a well known result of Wachs \cite[Theorem~4]{Wa89} (see also \cite[page~209]{GR93}). \Cref{subsec: involution} studies Eulerian and fix--Euler--Mahonian distributions on involutions and their colored analogues and generalizes a formula of D\'esarm\'enien and Foata \cite[Equation~(1.8)]{DF85} and Gessel and Reutenauer \cite[Equation~(7.3)]{GR93} (see also \cite[Section~5]{GZ20}) and a formula of Athanasiadis \cite[Equation~(40)]{Ath18}. Lastly, \cref{subsec: bimahonian} studies bimahonian and multivariate distributions, involving Eulerian and Mahonian statistics, on colored permutations. In what follows, we use the color order for colored permutation statistics and therefore omit the total order subscript in statistics and quasisymmetric functions.

%
%
\subsection{Euler--Mahonian identities}
\label{subsec: Euler--Mahonian identities}

	\Cref{eq: applicationIcolorChowGesselmaj} is Biagioli--Zeng's identity \eqref{eq: (des, maj) length} and \cref{eq: applicationIcolormajtypeA} appears in Assaf's work \cite[Equation~(13)~for~$t=1$]{Ass10}, where the author uses the length order.  
\begin{corollary}
\label{cor: applicationIcolor}
We have 
\begin{equation}
\label{eq: applicationIcolormajtypeA}
\sum_{w \in \fS_{n, r}} \, q^{\maj(w)} \ = \ r^n [n]_q!
\end{equation}
and
\begin{align}
\sum_{m \ge 0} \, ( [m + 1]_q + (r-1) [m]_q)^n \, x^m  \ &= \ \frac{\sum_{w \in \fS_{n,r}} \, x^{\des(w)} q^{\maj(w)}}{(x; q)_{n+1}} \label{eq: applicationIcolorChowGesselmaj} \\ 
\sum_{m \ge 0} \, \ r^n [m+1]_q^n \, x^m \ &= \ \frac{\sum_{w \in \fS_{n,r}} \, x^{\des^* (w)} q^{\maj(w)}}{(x; q)_{n+1}}. \label{eq: applicationIcolorChowGesselmajtypeA}
\end{align}
\end{corollary}
\begin{proof}
Specializing Equation \eqref{eq: F(Snr)} as in \cref{thm: Icolor} yields
\begin{align*}
\ps^{(r)} ( F (\fS_{n, r} ) ) \ &= \  \left(\frac{r}{(q)_1}\right)^n \\
\ps_m^{(r)} ( F ( \fS_{n, r} ) ) \ &= \ ( [m]_q + (r-1)[m-1]_q )^n \\
\widetilde{\ps}_m^{(r)} (F(\fS_{n, r})) \ &= \  r^n [m]_q^n.
\end{align*}
The proof follows by substituting in \cref{cor: Icolor} for $\cC = \fS_{n, r}$.
\end{proof}

	\Cref{eq: applicationIIcolorfmaj} computes the distribution of the flag major index (see \cref{eq: poincare series of Snr}) and \cref{eq: applicationIIcolordesfmaj} is Biagioli--Caselli, Chow--Mansour and Biagioli--Zeng's  identity \eqref{eq: (des, fmaj) color}.
\begin{corollary}
\label{cor: applicationIIcolor}
We have 
\begin{equation}
\label{eq: applicationIIcolorfmaj}
\sum_{w \in \fS_{n, r}} \, q^{\fmaj(w)} \ = \ [r]_q[2r]_q\cdots[nr]_q
\end{equation}
and
\begin{align}
\sum_{m \ge 0} \, [rm+1]_q^n \, x^m \ &= \ \frac{\sum_{w \in \fS_{n,r}} \, x^{\des(w)} q^{\fmaj(w)}}{(x; q^r)_{n+1}} \label{eq: applicationIIcolordesfmaj} \\ 
\sum_{m \ge 0} \, [r(m+1)]_q^n \, x^m \ &= \ \frac{\sum_{w \in \fS_{n,r}} \, x^{\des^*(w)} q^{\fmaj(w)}}{(x; q^r)_{n+1}}. \label{re: applicationIIcolordes*fmaj}
\end{align}
\end{corollary}
\begin{proof}
Specializing Equation \eqref{eq: F(Snr)} as in \cref{thm: IIcolor} yields
\[
\psi^{(r)} ( F(\fS_{n, r})) \ = \  \frac{1}{(q)_1^n} 
\]
and
\begin{align*}
\psi_m^{(r)} ( F(\fS_{n, r}) ) \ &= \ (1 + q + \cdots + q^{r-1} + q^r + q^{r+1} + \cdots \\ 
&+ q^{2r - 1} + q^{2r} + \cdots + q^{(m-1)r - 1} + q^{(m-1)r})^n \\
&= \ [r(m-1) +1]_q^n  
\end{align*}
and
\[
\widetilde{\psi}_m^{(r)} ( F(\fS_{n, r}) ) \ = \ [rm]_q^n.
\]
The proof follows by substituting in \cref{cor: IIcolor} for $\cC = \fS_{n, r}$.
\end{proof}

	\Cref{eq: applicationIIIcolorfdesfmaj} is Bagno and Biagioli's identity \eqref{eq: (fdes, fmaj) color}.
\begin{corollary}
\label{cor: applicationIIIcolor}
We have 
\begin{equation}
\label{eq: applicationIIIcolorfdesfmaj}
\sum_{m \ge 0} \, [m+1]_q^n \, x^m   \ = \ \frac{\sum_{w \in \fS_{n, r}} \, x^{\fdes(w)} q^{\fmaj(w)}}{(1-x)(1 - x^r q^r)(1 - x^r q^{2r}) \cdots (1 - x^r q^{nr})}.
\end{equation}
\end{corollary}
\begin{proof}
Specializing Equation \eqref{eq: F(Snr)} as in \cref{thm: IIIcolor} yields
\[
\phi_m^{(r)} (F(\fS_{n, r})) \ = \ (1 + q + \cdots + q^{m-1})^n = [m]_q^n.
\]
Notice that, from the definition of $\phi_m^{(r)}$, the summand $q^{m - 1}$ in the expression above comes from the substitution $x_{m - c + 1}^{(c-1)}$, according to whether $m \equiv c \pmod{r}$, for $c = 1, 2, \dots, r $, respectively. The proof follows by substituting in \cref{eq: IIIcolorfdesfmajF} for $\cC = \fS_{n, r}$.
\end{proof}

	The following theorem computes a formula for the $\fmaj_{k, \ell}$-distribution on $\fB_n$, which refines the $\fmaj$-distribution of Adin and Roichman \cite[Theorem~2]{AR01}. The right-hand side of \cref{eq: applicationIVfmaj} for $(k, \ell) = (2, 1)$, appears in this form in \cite[Theorem~3.5~for~$t=1$]{Fir05}. Furthermore, it computes Euler--Mahonian identities on $\fB_n$ for the pairs $(\des, \fmaj_{k, \ell})$ and $(\des^*, \fmaj_{k, \ell})$. \Cref{eq: applicationIVdesfmaj} reduces to Chow and Gessel's formulas \cite[Equation~(26)]{CG07} and \cite[Theorem~3.7]{CG07} for $(k, \ell) = (1, 0) $ and $(k, \ell) = (2, 1)$, respectively.
\begin{corollary}
\label{cor: applicationIVkl}
We have 
\begin{equation}
\label{eq: applicationIVfmaj}
\sum_{w \in \fB_n} \, q^{\fmaj_{k, \ell}(w)} \ = \  (1 + q^\ell)^n [n]_{q^k}!,
\end{equation}
and
\begin{align}
\sum_{m \ge 0} \, ([m+1]_{q^k} + q^\ell[m]_{q^k})^n \, x^m  \ &= \ \frac{\sum_{w \in \fB_n} \, x^{\des (w)} q^{\fmaj_{k, \ell}(w)}}{(x; q^k)_{n+1}} \label{eq: applicationIVdesfmaj}  \\
\sum_{m \ge 0} \, \ (1 + q^\ell)^n [m+1]_{q^k}^n \, x^m  \ &= \ \frac{\sum_{w \in \fB_n} \, x^{\des^* (w)} q^{\fmaj_{k, \ell}(w)}}{(x; q^k)_{n+1}}. \label{eq: applicationIVdes*fmaj}
\end{align}
\end{corollary}
\begin{proof}
 The proof follows by specializing Equation \eqref{eq: F(Snr)} for $r=2$ as in \cref{thm: IVklvariation} and substituting in \cref{cor: IVklvariation} for $\bB = \fB_n$.
\end{proof}
%

%
%
\subsection{Derangements}
\label{subsec: derangem}

	Permutations in $\fS_n$ without fixed points are called \emph{derangements}. Let $\dD_n$ be the set of all derangements in $\fS_n$. For a positive integer $n$, let
\[
D_n (x, q) \ := \ \sum_{w \in \dD_n} \, x^{\des(w)} q^{\maj(w)}
\]
be the $n$-th \emph{$(x, q)$-derangement~polynomial} and $d_n (q) := D_n(1, q)$ the \emph{$q$-derangement~numbers}. The $q$-derangement numbers satisfy (recall Formula \eqref{eq: Wachs} from \cref{sec: intro}).
\begin{equation}
\label{eq: Wachsagain}
d_n (q) \ = \ [n]_q! \sum_{k = 0}^n \, (-1)^k \, \frac{q^{\binom{k}{2}}}{[k]_q!}.
\end{equation}
as proved bijectively by Wachs \cite[Theorem~4]{Wa89} and later by Gessel and Reutenauer \cite[page~209]{GR93}. Eulerian and Mahonian distributions on derangements have been studied by many authors (see, for example, \cite[Section~2.1.4]{Ath18} and references therein). For nonnegative integers $0 \le k \le n$, let
\[
\binom{n}{k}_q \ := \ \dfrac{[n]_q!}{[k]_q![n-k]_q!}
\]
be a $q$-binomial coefficient. The following theorem provides an Euler--Mahonian identity on derangements in $\fS_n$, which refines Wachs' formula \eqref{eq: Wachsagain}.
\begin{theorem}
\label{thm: Euler--Mahonianderangements}
For a positive integer $n$, we have 
\begin{equation}
\label{eq: Euler--Mahonianderangements}
\sum_{m \ge 0} \, \sum_{k = 0}^n \, (-1)^k q^{\binom{k}{2}} \binom{m+1}{k}_q \, [m+1]_q^{n - k} \, x^m \ = \ \frac{D_n(x, q)}{(x; q)_{n+1}}.
\end{equation}
\end{theorem} 
The idea of the proof is to take the principal specialization of order $m$ of $F(\dD_n; \bx)$ and apply Formula \eqref{eq: psmF(A;x)} for $\aA = \dD_n$. Before the proof of \cref{thm: Euler--Mahonianderangements}, we need to recall a few facts.  Gessel and Reutenauer \cite[Theorem~8.1]{GR93} computed $F(\dD_n; \bx)$, as follows
\begin{equation}
\label{eq: GesselReutenauer}
F(\dD_n; \bx) \ = \ \sum_{k = 0}^n \, (-1)^k e_k(\bx) h_1(\bx)^{n-k},
\end{equation}
where $e_k$ (resp. $h_k$) is the \emph{$k$-th~elementary} (resp. \emph{complete~homogeneous}) symmetric function on $\bx$. The principal specializations of order $m$ of these symmetric functions are given by \cite[Proposition~7.8.3]{StaEC2}
\begin{align}
\ps_m(e_k(\bx)) \ &= \ q^{\binom{k}{2}} \binom{m}{k}_q  \label{eq: psme(x)}\\ 
\ps_m(h_k(\bx)) \ &= \ \binom{m + k - 1}{k}_q. \label{eq: psmh(x)}
\end{align}
We are now ready to give the proof of \cref{thm: Euler--Mahonianderangements}.

\begin{proof}[Proof of \cref{thm: Euler--Mahonianderangements}]
Setting $\aA = \dD_n$, Formula \eqref{eq: psmF(A;x)} becomes
\begin{equation}
\label{eq: derangementhelp}
\sum_{m \ge 1} \, \ps_m (F(\dD_n; \bx)) \, x^{m -1}  \ = \ \frac{D_n(x, q)}{(x; q)_{n +1}}.
\end{equation}
Taking the principal specialization of order $m$ of Formula \eqref{eq: GesselReutenauer} and compute, using Equations \eqref{eq: psme(x)} and \eqref{eq: psmh(x)}, yields
\begin{equation}
\label{eq: derangementhelpp}
\ps_m(F(\dD_n; \bx)) \ = \ \sum_{k = 0}^n \, (-1)^k q^{\binom{k}{2}} \binom{m}{k}_q [m]_q^{n - k}.
\end{equation}
The proof follows by substituting \cref{eq: derangementhelpp} in \eqref{eq: derangementhelp}.
\end{proof}

	Another proof of \cref{eq: Wachsagain} can be obtained by considering the stable principal specialization of Formula \eqref{eq: GesselReutenauer} instead and following the steps of the previous proof, as done by Gessel and Reutenauer \cite[Theorem~8.4]{GR93}. Next we consider a colored analogue of \cref{thm: Euler--Mahonianderangements} on colored permutation groups, which appears to be new even in the case $r=2$, of the hyperoctahedral group $\fB_n$.

	An element of $\fS_{n, r}$ without fixed points of zero color is called a \emph{colored~derangement}. Let $\dD_n^r$ be the set of all colored derangements in $\fS_{n, r}$. Faliharimalala and Zeng \cite[Equation~(2.7)]{FZ08} (see also \cite[Theorem~2.1]{Ass10}, where colored derangements are called cyclic derangements) proved the following formula
\begin{equation}
\label{eq: FaliharimalalaZeng}
|\dD_n^r| \ = \ r^n n! \sum_{k = 0}^n \frac{(-1)^k}{r^k k!},
\end{equation}
which generalizes the well known formula \cite[Equation~(2.11)]{StaEC1} for the number of derangements in the symmetric group $\fS_n$. In a subsequent paper \cite[Equation~(2.5)]{FZ11}, where the authors use the color order, they provide a formula for the colored $q$-derangement numbers 
\begin{equation}
\label{eq: FaliharimalalaZengfmaj}
\sum_{w \in \dD_n^r} \, q^{\fmaj(w)} \ = \ 
 [r]_q[2r]_q \cdots [nr]_q \, \sum_{k = 0}^n \, (-1)^k \, \frac{q^{r\binom{k}{2}}}{[r]_q[2r]_q \cdots [kr]_q},
\end{equation}
which reduces to Wachs' formula \eqref{eq: Wachsagain} for $r=1$ and generalizes a formula of Chow \cite[Theorem~5]{Cho06} (see also \cite[Equation~(6.8)~for~$Z = 1$]{FH08}) for $r=2$. The following theorem refines \cref{eq: FaliharimalalaZengfmaj}, by providing an Euler--Mahonian identity on colored derangements and reduces to \cref{thm: Euler--Mahonianderangements} for $r=1$.
\begin{theorem}
\label{thm: Euler--Mahoniancoloredderangements}
For a positive integer $n$, we have
\begin{equation}
\label{eq: Euler--Mahoniancoloredderangements}
\sum_{m \ge 0} \, \sum_{k = 0}^n \, (-1)^k q^{r\binom{k}{2}} \binom{m+1}{k}_{q^r} \, [rm + 1]_q^{n - k} \, x^m  \ = \ \frac{\sum_{w \in \dD_n^r} \, x^{\des(w)} q^{\fmaj(w)}}{(x; q^r)_{n+1}}. 
\end{equation}
\end{theorem}
\begin{proof}
Adin et al. \cite[Theorem~7.3]{AAER17} recently proved a signed analogue of \cref{eq: GesselReutenauer}. The following colored analogue
\begin{equation}
\label{eq: coloredgesselreutenauer}
F(\dD_n^r; \bX^{(r)}) \ = \ 
\sum_{k = 0}^n \, (-1)^k e_k(\bx^{(0)}) h_1(\bX^{(r)})^{n -k},
\end{equation}
where
\[
h_1(\bX^{(r)}) \ := \ x_1^{(0)} + x_2^{(0)} + \cdots + x_1^{(r-1)} + x_2^{(r-1)} + \cdots
\]
of \cref{eq: GesselReutenauer}, holds by trivially generalizing Adin et al.'s argument in the proof of \cite[Theorem~7.3]{AAER17} for general $r$ and using \cite[Theorem~16]{Poi98}. 
Specializing \cref{eq: coloredgesselreutenauer} as in \cref{thm: IIcolor} and compute, using Equation \eqref{eq: psme(x)} and the calculation in the proof of \cref{cor: applicationIIcolor}, yields 
\begin{equation}
\label{eq: coloredderangementhelp}
\psi_m^{(r)} ( F ( \dD_n^r) ) \ = \ \sum_{k = 0}^n \, (-1)^k q^{r\binom{k}{2}} \binom{m}{k}_{q^r} \, [r(m - 1) + 1]_q^{n - k}.
\end{equation}
The proof follows by substituting \cref{eq: coloredderangementhelp} in \cref{eq: IIcolordesfmajF} for $\cC = \dD_n^r$.
\end{proof}
	Another proof of \cref{eq: FaliharimalalaZengfmaj} can be obtained by considering the $\psi^{(r)}$ specialization  of \cref{eq: coloredgesselreutenauer} and substituting in \cref{cor: IIcolor} for $\cC = \dD_n^r$. Furthermore, taking the $\ps^{(r)}$ specialization instead, yields the following $\maj$-distribution on colored derangements, as computed by Assaf in \cite[Theorem~3.2~for~$t=1$]{Ass10}.
\begin{corollary}
\label{cor: assaf}
For a positive integer $n$, we have
\begin{equation}
\label{eq: assaf}
\sum_{w \in \dD_n^r} \, q^{\maj(w)} \ = \ r^n [n]_q! \sum_{k = 0}^n \, (-1)^k \frac{q^{\binom{k}{2}}}{r^k [k]_q!}.
\end{equation}
\end{corollary}
\begin{proof}
Specializing \cref{eq: coloredgesselreutenauer} as in \cref{thm: Icolor} yields
\[
\ps^{(r)} (F(\dD_n^r)) \ = \ \sum_{k = 0}^n \, (-1)^k  \frac{q^{\binom{k}{2}}}{(q)_k} \left(\frac{r}{(q)_1}\right)^{n -k}.
\]
The proof follows by substituting in \cref{eq: IcolormajF} for $\cC = \dD_n^r$.
\end{proof}

	The following Euler--Mahonian identity on signed derangements for the pair $(\des, \fmaj_{k, \ell})$ can be obtained by taking the $\theta$ specialization on \cref{eq: coloredgesselreutenauer} for $r=2$ and substituting in \cref{cor: applicationIVkl} for $\bB = \dD_n^2$
\begin{equation}
\label{eq: desklmajoronsignedderangements}
\sum_{m \ge 0} \, \sum_{i = 0}^n \, (-1)^i q^{k \binom{i}{2}} \binom{m+1}{i}_{q^k} ([m+1]_{q^k} + q^\ell [m]_{q^k})^{n - i} \, x^m \ = \ \frac{\sum_{w \in \dD_n^B} \, x^{\des(w)} q^{\fmaj_{k, \ell}(w)}}{(x; q^k)_{n+1}}.
\end{equation}
\Cref{eq: desklmajoronsignedderangements} coincides with \cref{eq: Euler--Mahoniancoloredderangements}  for $r=2$ and $(k, \ell) = (2, 1)$ and refines Chow's formula \cite[Theorem~5]{Cho06}. One can also compute the $\fmaj_{k, \ell}$-distribution on signed derangements
\begin{equation}
\label{eq: klfmajonsignedderangements}
\sum_{w \in \dD_n^B} \, q^{\fmaj_{k, \ell}(w)} \ = \
(1 + q^\ell)^{n} [n]_{q^{k}}!  \, \sum_{i = 0}^{n} \, (-1)^{i} \, \frac{q^{k\binom{i}{2}}}{(1 + q^\ell)^i [i]_{q^k}!}.
\end{equation}
\Cref{eq: klfmajonsignedderangements} reduces to Chow's formula \cite[Theorem~5]{Cho06}.

%
%
\subsection{Involutions and absolute involutions}
\label{subsec: involution}

Permutations in $\fS_n$ which consist only of one-cycles and two-cycles, when written in cycle notation, are called \emph{involutions}. Let $\iI_n$ be the set of all involutions in $\fS_n$. For a positive integer $n$, let 
\[
I_n (x, q, p) \ := \ \sum_{w \in \iI_n} \, x^{\des(w)} q^{\maj(w)} p^{\fix(w)},
\]
where $\fix(w)$ is the number of fixed points of $w$ and set  $I_n (x, q) := I_n (x, q, 1)$. This polynomial was considered by D\'esarm\'enien and Foata in \cite[Section~6]{DF85} and later by Gessel and Reutenauer in \cite[Section~7]{GR93}, where they computed a generating function for $I_n (x, q, p)$. In particular, D\'esarm\'enien and Foata proved \cite[Equation~(6.2)]{DF85} (where $(q; q)_n$ is to be replaced by $(t; q)_{n +1}$)
\begin{equation}
\label{eq: DesarmenienFoata}
\sum_{n \ge 0} \, \frac{I_n (x, q, p)}{(x; q)_{n +1}} \, z^n \ = \ 
\sum_{m \ge 0} \, (pz; q)_{m+1}^{-1} \prod_{0 \le i < j \leq m} \, (1 - z^2 q^{i + j})^{-1} \, x^m,
\end{equation}
where $I_0 (x, p, q) := 1$. 

One can prove \cref{eq: DesarmenienFoata} by taking the principal specialization of order $m$ of the quasisymmetric generating function for involutions according to fixed points \cite[Equation~(7.1)]{GR93} 
\begin{equation}
\label{eq: quasisymmetric generating function for involutions}
\sum_{n \ge 0} \sum_{w \in \iI_n} \, F_{n, \Des(w)}(\bx) \, p^{\fix(w)} z^{n} \ = \ \prod_{i \ge 1}(1 - pzx_{i})^{-1} \, \prod_{1 \le i < j}(1 - z^{2}x_{i}x_{j})^{-1}
\end{equation}
This is essentially the approach of Gessel and Reutenauer in the proof of \cite[Equation~(7.2)]{GR93}. The connecting link between the D\'esarm\'enien--Foata and Gessel--Reutenauer approaches is \cref{eq: schurtofun}. To be more specific, one has \cite[Corollary~7.13.8~and~Exercise~7.28]{StaEC2}
\begin{equation}
\label{eq: exercise7.28}
\sum_{\lambda} \, s_{\lambda}(\bx) \, p^{\c(\lambda)} z^{|\lambda|} \ = \
\prod_{i \ge 1}(1 - pzx_{i})^{-1} \, \prod_{1 \le i < j}(1 - z^2x_{i}x_{j})^{-1},
\end{equation}
where the sum runs through all partitions $\lambda$, $\c(\lambda)$ is the number of columns of $\lambda$ of odd length and $|\lambda|$ is the sum of all parts of $\lambda$ and  therefore \cref{eq: quasisymmetric generating function for involutions} follows from \cref{eq: exercise7.28}, together with \cref{eq: schurtofun} and the fact that 
the Robinson--Schensted correspondence restricts to a $\des$-preserving bijection between the set of involutions of $\fS_n$ and the set of all standard Young tableaux of size $n$.

An Euler--Mahonian identity on involutions involving the "hook-content formula" for Schur functions can be derived in the following way. Recall from \cite[Example~1~of~Section~3]{MacSFHP} the following notation
\[
\binom{n}{\lambda}_q := \prod_{u \in \lambda} \, \frac{1 - q^{n - \c(u)}}{1 - q^{\h(u)}},
\]
slightly altered to match our notation, where for a cell $u \in \lambda$, $\c(u)$ and $\h(u)$  is the \emph{content} and the \emph{hook length} of $u$, respectively. We refer to \cite[Section~7.21]{StaEC2} for more details on these concepts. 
Then, Macdonald's interpretation of Stanley's "hook-content formula" \cite[Theorem~7.21.2]{StaEC2} becomes
\begin{equation}
\label{eq: macdonald}
\ps_m ( s_{\lambda} (\bx) ) \ = \ q^{\b(\lambda)} \binom{m}{\lambda'}_q,
\end{equation}
where $\b (\lambda) := \sum_{i \ge 1} (i -1) \lambda_i$, for a partition $\lambda = (\lambda_1, \lambda_2, \dots)$ and $\lambda'$ denotes the conjugate partition to $\lambda$. Taking the principal specialization of order $m$ on 
\[
F ( \iI_n; \bx) \ = \ \sum_{ \lambda \vdash n} \, s_{\lambda} (\bx),
\]
applying Formula \eqref{eq: psmF(A;x)} for $\aA = \iI_n$ and using \cref{eq: macdonald} yields the following Euler--Mahonian identity on $\iI_n$
\begin{equation}
\label{eq: Euler--Mahonian identity on involutions}
\sum_{m \ge 0} \, \sum_{\lambda \vdash n} \, q^{\b(\lambda)} \binom{m+1}{\lambda'}_q \, x^m \ = \ \frac{I_n (x, q)}{(x; q)_{n +1}}.
\end{equation}

	We discuss colored analogues of Equations \eqref{eq: DesarmenienFoata} and \eqref{eq: quasisymmetric generating function for involutions} and \eqref{eq: Euler--Mahonian identity on involutions}. We deal with two types of involutions in colored permutation groups, the colored involutions and the absolute involutions. A \emph{colored~(resp.~absolute)~involution} is an element $w \in \fS_{n, r}$, such that $w^{-1} = w$ (resp. $\ol{w}^{-1} = w$). Let $\iI_n^r$ (resp. $\iI_n^{\abs}$) be the set of all colored (resp. absolute) involutions in $\fS_{n, r}$. Absolute involutions do not coincide with colored involutions for $r \ge 3$. For example, the colored permutation $3^1 2^0 1^3 4^2 6^3 5^1 \in \fS_{6, 4}$ is an involution, but not an absolute involution and on the other hand the colored permutation $3^1 2^0 1^1 4^2 6^3 5^3 \in \fS_{6, 4}$ is an absolute involution, but not an involution. 
	
	 Chow and Mansour \cite[Section~4]{CM10} studied colored involutions. In a similar fashion, if $w$ is an absolute involution with color vector $(c_1, c_2, \dots, c_n)$, then we see that
\begin{itemize}
\item $w \in \iI_n$ and
\item if $w(i) = j$, then $c_{w(i)} = c_j$ computed modulo $r$,
\end{itemize}
for some $i, j \in [n]$. Modifying the arguments of Chow and Mansour \cite[Proposition~7]{CM10} yields the following formulas
\begin{align*}
|\iI_n^{\abs}| \ &= \ r^n n! \sum_{k = 0}^{\lfloor n/2 \rfloor} \, \frac{(1 / 2r)^k}{k! (n - 2k)!} \\
\sum_{n \ge 0} \, |\iI_n^{\abs}| \, \frac{x^n}{n!} \ &= \ e^{r(x^2 / 2 + x)},
\end{align*}
where $|\iI_0^{\abs}| := 1$ and the following recurrence formula for the number of absolute involutions in $\fS_{n, r}$, 
\[
|\iI_{n+1}^{\abs}| \ = \ r (|\iI_n^{\abs}| + n |\iI_{n-1}^{\abs}|),
\]
for every positive integer $n \ge 1$, with initial condition $|\iI_1^{\abs}| = r$. 

	A polynomial $f(x)$ with real coefficients is called $\gamma$-positive if 
\[
	f(x) \ = \ \sum_{i=0}^{\lfloor n/2 \rfloor} 
\gamma_i x^i (1+x)^{n-2i},
\]
 for some $n \in \NN$ and nonnegative reals $\gamma_0, \gamma_1,\dots,\gamma_{\lfloor n/2 \rfloor}$. Chow and Mansour implicitly proved \cite[Proposition~8]{CM10} that the generating polynomial of the $\exc$-statistic on colored involutions is $\gamma$-positive for every even color $r$, where $\exc(w)$ is the number of excedances of $w$, that is indices $i \in [n]$, such that  $w(i) > i$, or $w(i) = i$ and $c_i > 0$ for any colored permutation $w$ with color vector $(c_1, c_2, \dots, c_n)$.	Recall from \cite[Theorem~15]{Stei94} that $\exc$ is Eulerian on colored permutations. Although it is not related to our study of Euler--Mahonian distributions, we record it here because of its own importance. For every even color $r$, \cite[Proposition~8]{CM10} states that
\begin{equation}
\label{eq: a lone gamma-positivity result}
\sum_{w \in \iI_n^r} \, x^{\exc(w)} \ = \ \sum_{i = 0}^{\lfloor n/2 \rfloor} \, r^i \gamma_{n, i} \, x^i (1 + x)^{n - 2i},
\end{equation}
where $\gamma_{n, i}$ is the number of $w \in \iI_n$ having $i$ number of  two-cycles. Gamma-positivity is a property that implies symmetry and unimodality and appears often in combinatorics. For more information we refer the reader to Athanasiadis' comprehensive survey \cite{Ath18}.
		
	In fact, one can further argue as in \cite[Proposition~8]{CM10} and prove the following
\[
\sum_{w \in \iI_n^{\abs}} \, x^{\exc (w)} \ = \ 
\sum_{i = 0}^{\lfloor n/2 \rfloor} \, r^i \gamma_{n, i} \, x^i (1 + (r-1)x)^{n -2i},
\]
where $\gamma_{n, i}$ as in \cref{eq: a lone gamma-positivity result}. The above mentioned formulas coincide with the corresponding formulas of Chow and Mansour \cite{CM10} for $r \le 2$.

	Absolute involutions appeared in Adin, Postnikov and Roichman's study \cite{APR10} of Gelfand models for colored permutation groups $\fS_{n, r}$. They are suitable for providing a colored analogue of D\'esarm\'enien and Foata's Formula \eqref{eq: DesarmenienFoata}. In particular, the colored Robinson--Schensted correspondence restricts to a $\des$-preserving bijection between $\iI_n^{\abs}$ and $\SYT_{n,r}$, the set of all standard Young $r$-partite tableaux of size $n$. In addition, from its description (see, for example, \cite[Section~5]{APR10}), the number of fixed points of color $j$ of an absolute involution in $\fS_{n, r}$ is equal to the number of odd columns of the $j$th part of the $P$-tableau, which corresponds to $w$ via the colored Robinson--Schensted correspondence. For a positive integer $n$, let
\[
F (\iI_n^{\abs}; \bX^{(r)}, p_0, p_1 \dots, p_{r-1}) \ := \ \sum_{w \in  \iI_n^{\abs}} \, F_w ( \bX^{(r)}) \, p_0^{\fix^0(w)} p_1^{\fix^1(w)} \cdots p_{r-1}^{\fix^{r-1}(w)},
\]
be the quasisymmetric generating function for absolute involutions according to fixed points of various colors, where $\fix^j (w)$ is the number of fixed points of $w \in \fS_{n, r}$ of color $j$. \Cref{eq: quasisymmetric generating function for absolute involutions} reduces to \cref{eq: quasisymmetric generating function for involutions} for $r=1$ and $p_0 = p$.
\begin{theorem}
\label{thm: absolute involutions}
We have 
\begin{equation}
\label{eq: quasisymmetric generating function for absolute involutions}
\sum_{n \ge 0} \, F (\iI_n^{\abs}; \bX^{(r)}, p_0, p_1 \dots, p_{r-1}) \, z^n \ = \
\prod_{c = 0}^{r-1} \, \prod_{i \ge 1} \, (1 - zp_c x_i^{(c)})^{-1} \, \prod_{1 \le i < j} \, (1 - z^2 x_i^{(c)} x_j^{(c)})^{-1}. 
\end{equation}
In particular, for a positive integer $n$ we have
\begin{equation}
\label{eq: F(Inr)}
F (\iI_n^{\abs}; \bX^{(r)}, p_0, p_1 \dots, p_{r-1}) \ = \ \sum_{\bl = (\lambda^{(0)}, \dots, \lambda^{(r-1)}) \vdash n} \, \prod_{j = 0}^{r-1} \, s_{\lambda^{(j)}}(\bx^{(j)}) \, p_j^{\c(\lambda^{(j)})}.
\end{equation}
\end{theorem}
\begin{proof}
The discussion before the statement of the theorem implies that
\[
\sum_{n \ge 0} \, F (\iI_n^{\abs}; \bX^{(r)}, p_0, p_1 \dots, p_{r-1}) \, z^n \ = \
\sum_{n \ge 0} \sum_{\bl = (\lambda^{(0)}, \dots, \lambda^{(r-1)}) \vdash n}
\sum_{\bQ \in \SYT(\bl)} \, F_{\bQ}(\bX^{(r)}) \prod_{j = 0}^{r-1} p_j^{\c(\lambda^{(j)})} z^n.
\]
Thus, applying the expansion \eqref{eq: schurtocoloredfun} yields
\begin{equation}
\label{eq: help^r}
\sum_{n \ge 0} \, F (\iI_n^{\abs}; \bX^{(r)}, p_0, p_1 \dots, p_{r-1}) \, z^n \ = \
 \sum_{\bl} \, \prod_{j = 0}^{r-1} \, s_{\lambda^{(j)}}(\bx^{(j)}) \, p_j^{\c(\lambda^{(j)})} z^{|\bl|},
\end{equation}
where the sum runs through all $r$-partite partitions $\bl =  (\lambda^{(0)}, \dots, \lambda^{(r-1)})$ and $|\bl| := |\lambda^{(0)}| + \cdots + |\lambda^{(r-1)}|$. Now, Formula \eqref{eq: F(Inr)} follows by extracting the coefficient of $z^n$ in \cref{eq: help^r} and Formula \eqref{eq: quasisymmetric generating function for absolute involutions} follows by expanding the right-hand side of \cref{eq: help^r} according to \cref{eq: exercise7.28} for every color.
\end{proof}

	Specializing \cref{eq: quasisymmetric generating function for absolute involutions} as in Theorens \ref{thm: IIcolor} and \ref{thm: IIIcolor} provides  colored analogues of D\'esarm\'enien and Foata's Formula \eqref{eq: DesarmenienFoata}. For a positive integer $n$, let
\begin{align*}
I_n (x, q, p_0, \dots, p_{r-1}) \ &:= \ 
\sum_{w \in \iI_n^{\abs}} \, x^{\des(w)} q^{\fmaj(w)} p_0^{\fix^0(w)} \cdots p_{r-1}^{\fix^{r-1}(w)} \\
I_n^{\flag} (x, q, p_0, \dots, p_{r-1}) \ &:= \ 
\sum_{w \in \iI_n^{\abs}} \, x^{\fdes(w)} q^{\fmaj(w)} p_0^{\fix^0(w)} \cdots p_{r-1}^{\fix^{r-1}(w)}
\end{align*}
and set $I_n (x, q) := I_n (x, q, 1, \dots, 1)$. The following corollary reduces to \cref{eq: DesarmenienFoata} for $r = 1, p_0 = p$.
\begin{corollary}
\label{cor: eulermahonian for colored involutions}
We have
\begin{align}
\sum_{n \ge 0} \, \frac{I_n (x, q, p_0, \dots, p_{r-1})}{(x; q^r)_{n +1}} \, z^n \ &= \ 
\sum_{m \ge 0} \, (p_0 z; q^r)_{m +1}^{-1} \prod_{0 \le i < j \le m} (1 - z^2 q^{r (i + j)})^{-1} \notag \\
&\qquad 
\prod_{c = 1}^{r-1} \, (p_cq^cz; q^r)_m^{-1} \, \prod_{0 \le i < j \le m-1} \, 
(1 - z^2 q^{r (i + j ) + 2c})^{-1} \, x^m \label{eq: colored DesarmenienFoata} 
\end{align}
and
\begin{align}
\sum_{n \ge 0} \, \frac{ [x]_r I_n^{\flag}(x, q, p_0, \dots, p_{r-1})}{(x^r; q^r)_{n +1}} \, z^n \ &= \ 
\sum_{m \ge 0} \, (p_0 z; q^r)_{\lfloor \frac{m}{r} \rfloor}^{-1}  \prod_{0 \le i < j \le \lfloor \frac{m}{r} \rfloor} (1 - z^2 q^{r (i + j)})^{-1} \notag \\
&\qquad 
\prod_{c = 1}^{r-1} (p_c z q^c; q^r)_{\lfloor \frac{m-1}{r} \rfloor}^{-1} \prod_{0 \le i < j \le \lfloor \frac{m-1}{r} \rfloor}
(1 - z^2 q^{r (i + j) + 2c})^{-1} \, x^m \label{eq: flagcolored DesarmenienFoata} 
\end{align}
where $I_0 (x, q, p_0, \dots, p_{r-1}) = I_0^{\flag} (x, q, p_0, \dots, p_{r-1}) := 1$.
\end{corollary}
\begin{proof}
The proof follows by specializing Formula \eqref{eq: quasisymmetric generating function for absolute involutions} as in Theorems \ref{thm: IIcolor} and \ref{thm: IIIcolor} and taking the generating function on $m$, changing the order of summation on the left-hand side and using Equations \eqref{eq: IIcolordesfmajF} and \eqref{eq: IIIcolorfdesfmajF} for $\cC = \iI_n^{\abs}$, respectively.
\end{proof}

	In the case $r=2$, \cref{eq: colored DesarmenienFoata} for $q = p_0 = p_1 = 1$ was proved in \cite[Theorem~1.1]{Mou19} and the first instance of this reasoning appeared in the work of Athanasiadis \cite[Proof~of~Proposition~2.22]{Ath18}. Next, we compute Euler--Mahonian identities on absolute involutions for the pairs $(\des, \fmaj)$ and $(\des, \maj)$, which reduce to \cref{eq: Euler--Mahonian identity on involutions} for $r =1$.
\begin{theorem}
\label{thm: euler--mahonian identities on absolute involutions}
For a positive integer $n$, we have 
\begin{align}
\sum_{m \ge 0} \sum q^{r\b (\bl) + \sum_{j=0}^{r-1} j|\lambda^{(j)}|} \binom{m+1}{\lambda^{(0)'}}_{q^r}\binom{m}{\lambda^{(1)'}}_{q^r}\cdots\binom{m}{\lambda^{(r-1)'}}_{q^r} x^m \ &= \ \frac{I_n^{\abs} (x, q)}{(x; q^r)_{n+1}}  \label{eq: desfmaj for absolute involutions} \\
\sum_{m \ge 0} \sum q^{\b (\bl)} \binom{m+1}{\lambda^{(0)'}}_q \binom{m}{\lambda^{(1)'}}_q \cdots \binom{m}{\lambda^{(r-1)'}}_q x^m \ &= \ \frac{\sum_{w \in \iI_n^{\abs}} x^{\des(w)} q^{\maj(w)}}{(x; q)_{n+1}},
\label{eq: desmaj for absolute involutions}
\end{align}
where the sums run through all $r$-partite partitions $\bl = (\lambda^{(0)}, \dots, \lambda^{(r-1)})$ of $n$ and $\b(\bl) := \b (\lambda^{(0)} + \cdots + \lambda^{(r-1)})$.
\end{theorem}
\begin{proof}
For Equation \eqref{eq: desmaj for absolute involutions}, specialize Equation \eqref{eq: F(Inr)} for $p_0= \cdots = p_{r-1} = 1$ as in \cref{thm: Icolor} and apply \cref{cor: Icolor} for $\cC = \iI_n^{\abs}$ and  \cref{eq: macdonald} for very color. For Equation \eqref{eq: desfmaj for absolute involutions}, we specialize as in \cref{thm: IIcolor} and apply \cref{cor: IIcolor} as before, but taking into account that
\[
s_{\lambda^{(j)}} (q^j, q^{r + j}, \dots, q^{r (m-2) + j}, 0) \ = \ q^{j |\lambda^{(j)}|} s_{\lambda^{(j)}} (1, q^r, \dots, q^{r (m-2)}, 0),
\]
for every color $1 \le j \le r-1$, due to the homogeneousness of Schur functions.
\end{proof}

	We finish  by providing a colored generalization of a formula due to Athanasiadis \cite[Proposition~2.2.]{Ath18} (see also \cite[Corollary~5.7~(b)]{GZ20}), which expresses the generating polynomials of the $\des$-distribution on $\iI_n$ and $\iI_n^2$  in terms of Eulerian polynomials $A_n (x)$ and $A_{n, 2}(x)$, respectively. Let $I_n^{\abs} (x) := I_n^{\abs} (x, 1)$. For $w \in \fS_{n, r}$, let $\c^j (w)$ be the number of colored cycles of $w$ of color $j$, for every $0 \le j \le r-1$. The following corollary reduces to \cite[Proposition~2.22]{Ath18} for $r \le 2$.
\begin{corollary}
\label{cor: colored athanasiadis}
For a positive integer $n$, we have 
\begin{equation}
\label{eq: colored athanasiadis}
I_n^{\abs} (x) \ = \ \frac{1}{r^n n!} \, \sum_{w \in \fS_{n, r}} \, (1 - x)^{n - \c^0 (w \overline{w})} \, A_{\c^0 (w \overline{w}), r} (x).
\end{equation}
\end{corollary}
\begin{proof}
From \cref{eq: IIcolordesfmajF}, for $\cC = \iI_n^{\abs}$ and $q =1$, we have
\begin{equation}
\label{eq: one}
\frac{I_n^{\abs} (x)}{(1 - x)^{n + 1}} \ = \ \sum_{m \ge 1} \, F (\iI_n^{\abs}; 1^m, 1^{m-1}, \dots, 1^{m-1}) \, x^{m -1},
\end{equation}
where $1^m$ in the above notation means that $x_1^{(0)}= x_2^{(0)} = \cdots = x_m^{(0)} = 1$ and $x_{m + 1}^{(0)} = x_{m + 2}^{(0)} = \cdots = 0$ etc.. Applying Equation \eqref{eq: F(Inr)} for $p_0 = p_1 = \cdots = p_{r-1} = 1$, \cref{eq: one} becomes
\begin{equation}
\label{eq: two}
\frac{I_n^{\abs} (x)}{(1 - x)^{n + 1}} \ = \
\sum_{m \ge 1} \sum \, s_{\lambda^{(0)}} (1^m) s_{\lambda^{(1)}} (1^{m-1}) \cdots s_{\lambda^{(r-1)}}(1^{m-1}) \, x^{m -1},
\end{equation}
where the second sum runs through all $r$-partite partitions $\bl = (\lambda^{(0)}, \dots, \lambda^{(r-1)})$ of  $n$. From the properties of the characteristic map, we know that 
\[
s_{\lambda^{(0)}} (1^m) s_{\lambda^{(1)}} (1^{m-1}) \cdots  s_{\lambda^{(r-1)}} (1^{m-1})= \ch (\chi^{\bl}) (1^m, 1^{m-1}, \dots, 1^{m-1})
\]
 and we can use \cref{eq: FrobDef} to expand it in the colored power sum basis, as follows 
\[
\ch_r (\chi^{\bl}) (1^m, 1^{m-1}, \dots, 1^{m-1}) \ = \ 
\frac{1}{r^n n!} \sum_{w \in \fS_{n, r}} \, \chi^{\bl} (w) p_{\ct (w)} (1^m, 1^{m-1}, \dots, 1^{m-1}),
\]
where $\chi^{\bl}$ is the irreducible $\fS_{n, r}$-character associated with the $r$-partite partition $\bl \vdash n$. But, 
\begin{align*}
p_{\ct (w)} (1^m, 1^{m-1}, \dots, 1^{m-1}) \ &= \ 
\prod_{j = 0}^{r-1} (m + (\zeta^j + \zeta^{2j} + \cdots + \zeta^{(r-1)j})(m-1))^{c^j(w)} \\
&= \ (r (m-1) + 1)^{c^0 (w)} \, \prod_{j = 1}^{r-1} (m - (m -1))^{c^j(w)} \\
&= (r (m-1) + 1)^{c^0 (w)},
\end{align*}
because 
\[
\zeta^j + \zeta^{2j} + \cdots + \zeta^{(r-1)j} \ = \ \frac{1 - \zeta^{jr}}{1 - \zeta^j} - 1 \ = \ -1,
\] 
for every $ 1 \le j \le r-1$. Combining these calculations, substituting in \cref{eq: two} and changing the order of summation yields
\begin{equation}
\label{eq: three}
\frac{I_n^{\abs} (x)}{(1 - x)^{n + 1}} \ = \
\frac{1}{r^n n!} \, \sum_{w \in \fS_{n, r}} \, \left(\sum_{\bl \vdash n} \, \chi^{\bl} (w) \right) \, \left(\sum_{m \ge 0} \, (rm +1)^{c^0 (w)} \, x^{m} \right).
\end{equation}
A special case of a well known result due to Frobenius and Schur (see, for example, \cite[Exercise~7.69~(c)]{StaEC2} and references therein) is that the sum of all irreducible $\fS_n$-characters computed in the conjugacy class corresponding to the cycle type of $w \in \fS_n$ is equal to the number of square roots of $w$ in $\fS_n$. Adin, Postnikov and Roichman \cite[Theorem~3.4]{APR10} extended this result to colored permutation groups by proving the following 
\begin{equation}
\label{eq: adin+postnikov+roichman}
\sum_{\bl \vdash n} \, \chi^{\bl} (w) \ = \ |\{ u \in \fS_{n, r} : u \overline{u} = w\}|,
\end{equation}
for every $w \in \fS_{n, r}$. Thus, the proof follows by substituting \cref{eq: adin+postnikov+roichman} in \cref{eq: three} and using Steingr\'{i}msson's Formula \eqref{eq: Steingrimsson}
\end{proof}
%

%
%
\subsection{Bimahonian and multivariate distributions}
\label{subsec: bimahonian}

	In what follows, the generating polynomials of permutation statistics are set to equal 1 for $n = 0$. For a permutation statistic $\stat$ and a colored permutation $w$, we write $\istat(w) := \stat(w^{-1})$ and $\ol{\istat}(w) := \stat (\ol{w}^{-1})$. A pair of statistics is called \emph{bimahonian} if it is equidistributed with $(\maj, \imaj)$. A celebrated result, due to Foata and Sch\"{u}tzenberger \cite[Theorem~1]{FS78}, states that the pair  $(\maj, \inv)$ is bimahonian on $\fS_n$. Gessel \cite[Theorem~8.5]{Ges77} computed the following generating function for the bimahonian statistic $(\inv, \maj)$
\begin{equation}
\label{eq: gessel bimahonian}
\sum_{n \ge 0} \, \frac{\sum_{w \in \fS_n} \, q^{\inv(w)} p^{\maj(w)}}{(q)_n(p)_n} \, z^n \ = \ \frac{1}{(z; q, p)_{\infty, \infty}},
\end{equation}
where  
\[
(z; q, p)_{\infty, \infty} \ := \prod_{i \geq 1} \prod_{j \ge 1} \, (1 - zq^{i-1}p^{j-1}). 
\]
\Cref{eq: gessel bimahonian} also holds for the bimahonian pair $(\maj, \imaj)$. This appeared implicitly in Gordon's work \cite{Gor63} and later made explicit by Roselle \cite{Ros74}. Because of that, \cref{eq: gessel bimahonian} for the bimahonian pair $(\maj, \imaj)$ is often called \emph{Roselle~identity}. Garsia and Gessel \cite{GG79} studied bieulerian-bimahonian  distributions, meaning the four-variate distribution $(\des, \ides, \maj, \imaj)$ and proved the following generating function 
\begin{equation}
\label{eq: garsia gessel}
\sum_{n \ge 0} \, \frac{\sum_{w \in \fS_n} \, x^{\des(w)} y^{\ides(w)} q^{\maj(w)} p^{\imaj(w)}}{(x; q)_{n+1} (y, p)_{n +1}} \, z^n \ = \ 
\sum_{m_1 \ge 0} \sum_{m_2 \ge 0} \, \frac{x^{m_1} y^{m_2}}{(z; q, p)_{m_1+1, m_2+1}},
\end{equation}
where
\[
(z; q, p)_{k, l} \ := \prod_{i = 1}^k \prod_{j = 1}^l \, (1 - zq^{i-1}p^{j-1}),
\]
for every positive integers $k, l$.

	Equations \eqref{eq: gessel bimahonian} and \eqref{eq: garsia gessel} can be proved by taking the stable principal specialization and the principal specialization of order $m$ of the following identity \cite[Equation~(7.114)~and~Equation~(7.44)]{StaEC2}
\begin{equation}
\label{eq: cauchy kernel}
\sum_{n \ge 0} \sum_{w \in \fS_n} \, F_{n, \Des(w)} (\bx) F_{n, \Des(w^{-1})} (\by) \, z^n \ = \ 
\prod_{i \ge 1} \prod_{j \ge 1} \, (1 - z x_i y_j )^{-1},
\end{equation}
and using Formulas \eqref{eq: psF(A;x)} and \eqref{eq: psmF(A;x)} for $\aA = \fS_n$, respectively. This is essentially the approach of \cite[Corollary~7.23.9]{StaEC2} (see also \cite{DF85}). This section develops a colored analogue of this approach and provides colored analogues of Equations \eqref{eq: gessel bimahonian} and \eqref{eq: garsia gessel} for bimahonian and bieulerian-bimahonian distributions on colored permutations.

	For every $0 \le j \le r-1$, let $\bx'^{(j)} = (x_1'^{(j)}, x_2'^{(j)}, \dots)$ be another sequence of commuting indeterminates and let $\bX'^{(r)} := (x_i'^{(0)}, x_i'^{(1)}, \dots, x_i'^{(r-1)})_{i \ge 1}$. The following lemma, essentially due to Poirier \cite[Lemma~4]{Poi98}, is a colored analogue of \cref{eq: cauchy kernel}.
\begin{lemma}
\label{lem: colored cauchy kernel}
We have 
\begin{equation}
\label{eq: colored cauchy kernel}
\sum_{n \ge 0} \sum_{w \in \fS_{n, r}} \, F_w (\bX^{(r)}) F_{\overline{w}^{-1}} (\bX'^{(r)}) \, z^n \ = \ 
\prod_{c = 0}^{r-1} \prod_{i \ge 1} \prod_{j \ge 1} \, (1 - z x_i^{(c)} x_j'^{(c)})^{-1}.
\end{equation}
\end{lemma}
\begin{proof}
Recall from \cite[Lemma~4]{Poi98} that
\begin{equation}
\label{eq: HELP}
\sum_{n \ge 0} \sum_{\bl \vdash n} s_{\bl} (\bX^{(r)}) s_{\bl} (\bX'^{(r)}) \, z^n \ = \
\prod_{c = 0}^{r-1} \prod_{i \ge 1} \prod_{j \ge 1} \, (1 - z x_i^{(c)} x_j'^{(c)})^{-1},
\end{equation}
The proof follows by expanding the products $s_{\bl} (\bX^{(r)}), s_{\bl} (\bX'^{(r)})$ in the left-hand side of \cref{eq: HELP} using \cref{eq: schurtocoloredfun} and then applying the colored Robinson--Schensted correspondence.
\end{proof}

	The following corollary computes generating functions for the bimahonian distribution $(\maj, \ol{\imaj})$ and for the bieulerian-bimahonian distributions $(\des, \ol{\ides}, \maj, \ol{\imaj})$ and $(\des^*, \ol{\ides^*}, \maj, \ol{\imaj})$ on colored permutations. Generating functions for the four-variate distribution $(\des, \ides, \maj, \imaj)$ on colored permutations have been proved by Biagioli and Zeng \cite[Theorem~7.1]{BZ11}, where the authors use the length order and by Reiner \cite[Corollary~7.3]{Rei93}, where he considers $i \in [n]$ to be a descent of $w \in \fS_{n, r}$, if $\ell (w s_i^{-1}) = \ell(w) - 1$, where $ s_1, s_2, \dots, s_{n-1}$ and $s_n := s_0$, as defined in \cref{subsec: colored}. For $r=2$, \cref{eq: des ides maj imaj} coincides with Biagioli and Zeng's formula.
\begin{corollary}
\label{cor: Imultivariate}
We have
\begin{equation}
\label{eq: maj imaj}
\sum_{n \ge 0} \, \frac{\sum_{w \in \fS_{n, r}} \, q^{\maj(w)} p^{\ol{\imaj}(w)}}{(q)_n (p)_n} \, z^n \ = \
\frac{1}{(z; q, p)_{\infty, \infty}^r},
\end{equation}
and 
\begin{align}
\sum_{n \ge 0} \, \frac{\sum_{w \in \fS_{n, r}} \, x^{\des(w)} y^{\ol{\ides}(w)} q^{\maj(w)} p^{\ol{\imaj}(w)}}{(x; q)_{n+1} (y; p)_{n+1}} \, z^n \ &= \
\sum_{m_1, m_2 \ge 0} \, \frac{x^{m_1} y^{m_2}}{(z; q, p)_{m_1 +1, m_2 + 1} (z; q, p)_{m_1, m_2}^{r-1}}  \label{eq: des ides maj imaj} \\
\sum_{n \ge 0} \frac{\sum_{w \in \fS_{n, r}} \, x^{\des^*(w)} y^{\ol{\ides^*}(w)} q^{\maj(w)} p^{\ol{\imaj}(w)}}{(x; q)_{n+1} (y; p)_{n+1}} \, z^n \ &= \
\sum_{m_1, m_2 \ge 0} \, \frac{x^{m_1} y^{m_2}}{(z; q, p)_{m_1 +1, m_2 + 1}^r}. \label{eq: des* ides* maj imaj}
\end{align}
\end{corollary} 
\begin{proof}
Combine \cref{cor: Icolor} for $\cC = \fS_{n, r}$ and \cref{lem: colored cauchy kernel} and the proof follows.
\end{proof}

	The following corollary computes the generating functions for the bimahonian distribution $(\fmaj, \ol{\ifmaj})$ and bieulerian-bimahonian distributions $(\des, \ol{\ides}, \fmaj, \ol{\ifmaj})$ and $(\des^*, \ol{\ides^*}, \fmaj, \ol{\ifmaj})$ on $r$-colored permutations. Generating functions for the bimahonian distribution $(\fmaj, \ifmaj)$ have been proved by Foata and Han \cite[Equation~(4.3)]{FH06} for $r=2$, where the authors use the color order, by Biagioli and Zeng \cite[Proposition~8.5]{BZ11}, where the authors use the length order and by Biagioli and Caselli \cite[Equation~(21)~for~$p=s=1$]{BC12}, where the authors use the color order. The latter also proved \cite[Proposition~6.2~for~$a_1=a_2=1$]{BC12} a generating function for the four-variate distribution $(\des, \ides, \fmaj, \ifmaj)$ on $r$-colored permutations. For $r=2$ \cref{eq: fmaj ifmaj} coincides with  Foata--Han, Biagioli--Zeng and Biagioli--Caselli's formulas and for the case $r \geq 3$, Equations \eqref{eq: fmaj ifmaj} and \eqref{eq: des ides fmaj ifmaj} are slightly different than  Biagioli and Caselli's formulas \cite[Equation~(21)~and~Proposition~6.2]{BC12}, respectively.

\begin{corollary}
\label{cor: IImultivariate}
We have
\begin{equation}
\label{eq: fmaj ifmaj}
\sum_{n \ge 0} \, \frac{\sum_{w \in \fS_{n, r}} \, q^{\fmaj(w)} p^{\ol{\ifmaj}(w)}}{(q^r)_n (p^r)_n} \, z^n \ = \
\prod_{c = 0}^{r-1} \, \frac{1}{(z (qp)^c; q^r, p^r)_{\infty, \infty}},
\end{equation}
and 

\begin{multline}
\sum_{n \ge 0} \, \frac{\sum_{w \in \fS_{n, r}} \, x^{\des(w)} y^{\ol{\ides}(w)} q^{\fmaj(w)} p^{\ol{\ifmaj}(w)}}{(x; q^r)_{n+1} (y; p^r)_{n+1}} \, z^n  = \\
 \sum_{m_1, m_2 \ge 0} \, \frac{x^{m_1} y^{m_2}}{(z; q^r, p^r)_{m_1 +1, m_2 + 1} \prod\limits_{c = 1}^{r-1} (z (qp)^c; q^r, p^r)_{m_1, m_2}}   \label{eq: des ides fmaj ifmaj}
\end{multline}
and
\begin{equation}
\label{eq: des* ides* fmaj ifmaj}
\sum_{n \ge 0} \, \frac{\sum_{w \in \fS_{n, r}} \, x^{\des^*(w)} y^{\ol{\ides^*}(w)} q^{\fmaj(w)} p^{\ol{\ifmaj}(w)}}{(x; q^r)_{n+1} (y; p^r)_{n+1}} \, z^n \ = \ \sum_{m_1, m_2 \ge 0} \, \frac{x^{m_1} y^{m_2}}{\prod\limits_{c = 0}^{r-1} (z (qp)^c; q^r, p^r)_{m_1 + 1, m_2 + 1}}.
\end{equation}
\end{corollary}
\begin{proof}
Combine \cref{cor: IIcolor} for $\cC = \fS_{n, r}$ and \cref{lem: colored cauchy kernel} and the proof follows.
\end{proof}
The following corollary computes a generating function for the bieulerian-bimahonian distribution $(\fdes, \ol{\ifdes}, \fmaj, \ol{\ifmaj})$ on $r$-colored permutations, which reduces to Foata and Han's formula \cite[Theorem~1.1]{FH06}, for $r=2$, where the authors use the color order.
\begin{corollary}
\label{cor: IIImultivariate}
We have
\begin{multline}
\sum_{n \ge 0} \, \frac{[r]_x [r]_y \, \sum_{w \in \fS_{n, r}} \, x^{\fdes(w)} y^{\ol{\ifdes}(w)} q^{\fmaj(w)} p^{\ol{\ifmaj}(w)}}{(x^r; q^r)_{n+1} (y^r; p^r)_{n+1}} \, z^n \ =  \\
  \sum_{m_1, m_2 \ge 0} \, \frac{x^{m_1} y^{m_2}}{(z; q^r, p^r)_{\lfloor \frac{m_1}{r} \rfloor +1, \lfloor \frac{m_2}{r} \rfloor +1} \prod\limits_{c = 1}^{r-1} (z(qp)^c; q^r, p^r)_{\lfloor \frac{m_1-1}{r} \rfloor +1, \lfloor \frac{m_2-1}{r} \rfloor +1}}.  \label{eq: fdes ifdes fmaj ifmaj} 
\end{multline}
\end{corollary}
\begin{proof}
Combine \cref{thm: IIIcolor} for $\cC = \fS_{n, r}$ and \cref{lem: colored cauchy kernel} and the proof follows.
\end{proof}
We close this section by providing generating functions for the bimahonian statistic $(\fmaj_{k, \ell}, \ol{\ifmaj_{k', \ell'}})$ and for the bieulerian-bimahonian statistic $(\des, \ol{\ides}, \fmaj_{k, \ell}, \ol{\ifmaj_{k', \ell'}})$
\begin{equation}
\sum_{n \ge 0} \, \frac{\sum_{w \in \fB_n} \, q^{\fmaj_{k, \ell}(w)} p^{\ol{\ifmaj_{k', \ell'}}(w)}}{(q^k)_n (p^{k'})_n} \, z^n \ = \
\frac{1}{(z; q^{k}, p^{k'})_{\infty, \infty} (z q^\ell p^{\ell'}; q^k, p^{k'})_{\infty, \infty}} \label{eq: fmajkl ifmajk'l'} 
 \end{equation}
and
\begin{multline} 
\sum_{n \ge 0} \, \frac{\sum_{w \in \fB_n} \, x^{\des(w)} y^{\ol{\ides}(w)} q^{\fmaj_{k, \ell}(w)} p^{\ol{\ifmaj_{k', \ell'}}(w)}}{(x; q^k)_{n+1} (y; p^{k'})_{n+1}} \, z^n  = \\ 
 \sum_{m_1, m_2 \ge 0} \, \frac{x^{m_1} y^{m_2}}{(z; q^k, p^{k'})_{m_1 +1, m_2 + 1} (z q^\ell p^{\ell'}; q^k, p^{k'})_{m_1, m_2}}, \label{eq: des ides fmajkl ifmajk'l'}
\end{multline}
for some positive (resp. nonnegative) integers $k, k'$ and $\ell, \ell'$, which reduce to Equations \eqref{eq: maj imaj}, \eqref{eq: fmaj ifmaj} and \eqref{eq: des ides maj imaj}, \eqref{eq: des ides fmaj ifmaj} for $r=2$ and $(k, \ell) = (k', \ell') = (1, 0)$, $(k, \ell) = (k', \ell') = (2, 1)$, respectively. Assigning different values for the pairs $(k, \ell)$ and $(k', \ell')$ in Formula \eqref{eq: fmajkl ifmajk'l'} results, for example, in generating functions for pairs of different Mahonian statistics on signed permutations such as $(\maj, \fmaj)$ (see also \cite[Equation~(5.12)]{BZ10} and \cite[Proposition~8.4]{BZ11}).

\section*{Acknowledgments}
The author would like to thank Christos Athanasiadis for sharing his ideas on specializations of quasisymmetric functions, as well as for his valuable comments on the presentation of this paper. This research is co-financed by Greece and the European Union (European Social Fund- ESF) through the Operational Programme "Human Resources Development, Education and Lifelong Learning" in the context of the project "Strengthening Human Resources Research Potential via Doctorate Research – 2nd Cycle" (MIS-5000432), implemented by the State Scholarships Foundation (IKY).

\end{document}